%% file: manuscript.tex
\documentclass{amsart}

\usepackage{fullpage}
\usepackage{amsfonts, amsmath, amssymb, amsthm}
\usepackage{booktabs}
\usepackage{caption}
\usepackage[normalem]{ulem}
\usepackage{xcolor}
\usepackage{tikz}
\usetikzlibrary{shapes,arrows,matrix}

\usepackage{physics}

\usepackage{algorithm}
\usepackage{algpseudocode}

\usepackage[export]{adjustbox}
\usepackage{subcaption}
\usepackage{array}
\usepackage[mathscr]{euscript}
\usepackage{nicefrac}
\usepackage{xspace}
\usepackage{esvect}
\usepackage{bbm}

\newtheorem{theorem}{Theorem}[section]

\newtheorem{remark}{Remark}
\newtheorem{lemma}[theorem]{Lemma}
\newtheorem{proposition}[theorem]{Proposition}
\theoremstyle{definition}
\newtheorem{definition}{Definition}[section]

\tikzset{
  block/.style={rectangle, draw, fill=blue!20, text width=10em, text centered, rounded corners, minimum height=4em},
  line/.style={draw, -latex'}
}

\def\mU{\mathbf{U}}

\def\vb{\mathbf{b}}

\def\Real{\mathbb{R}}
\def\mA{\mathbf{A}}

\def\mD{\mathbf{D}}
\def\mC{\mathbf{C}}

\def\pr{\mathbf{Pr}}

\def\vx{\mathbf{x}}

\def\vz{\mathbf{z}}

\def\ve{\mathbf{e}}

\def\E{\mathbb{E}}

\def\eps{\pmb{\varepsilon}}

\def\mM{\mathbf{M}}

\def\tr{\text{Tr}}

\def\id{\mathbf{I}}

\def\hess{\text{Hess}}

\def\det{\text{Det}}
\def\hom{\mathrm{h}}
\def\ihom{\mathrm{ih}}

\newcommand{\ql}[1]{\ifmmode{\textcolor{blue}{QL: #1}}\else {\textcolor{teal}{[QL: #1]}}\fi}

\usepackage[colorlinks=true]{hyperref}

\title{Optimal Design for Linear Models via Gradient Flow}

\author{Ruhui Jin}
\email{rjin18@jhu.edu}
\address{Department of Applied Mathematics and Statistics, John Hopkins University}

\author{Martin Guerra}
\email{mguerra4@wisc.edu}
\address{Department of Mathematics, University of Wisconsin–Madison}

\author{Qin Li}
\email{qinli@math.wisc.edu}
\address{Department of Mathematics, University of Wisconsin–Madison}

\author{Stephen J. Wright}
\email{swright@cs.wisc.edu}
\address{Department of Computer Sciences, University of Wisconsin–Madison}

\date{} 


\begin{document}

\input{sections/abstract}

\maketitle

\input{sections/1-intro}
\input{sections/2-problem}

\input{sections/3-method}

\input{sections/4-theory}
\input{sections/5-setup}

\input{sections/6-numerics}

\input{sections/7-conclusions}

\bibliographystyle{plain}

\input{manuscript.bbl}


\end{document}

%% file: sections/abstract.tex
\begin{abstract}
Optimal experimental design (OED) aims to choose the observations in an experiment to be as informative as possible, according to certain statistical criteria.
In the linear case (when the observations depend linearly on the unknown parameters), it seeks the optimal weights over rows of the design matrix $\mA$ under certain criteria. 
Classical OED assumes a discrete design space and thus a design matrix with finite dimensions.
In many practical situations, however, the design space is continuous-valued, so that the OED problem is one of optimizing over a continuous-valued design space. The objective becomes a functional over the probability measure, instead of a function of a finite dimensional vector.
This change of perspective requires a new set of techniques to optimize over probability measures, and Wasserstein gradient flow becomes a natural candidate. 
Both the first-order criticality and the convexity properties of the OED objective are presented. 
Computationally, the Monte Carlo particle method is used to translate the gradient flow equation formulation into a numerical algorithm. This algorithm is applied to two elliptic inverse problems.
\end{abstract}

%% file: sections/1-intro.tex
\section{Introduction}
\label{sec: intro}

The problem of inferring unknown parameters from measurements is ubiquitous in real-world engineering contexts, such as biological chemistry \cite{ABBCEKPR12}, medical imaging \cite{CLL18}, and climate science \cite{BGJS13,PMSG14}. 
This problem is termed  ``parameter identification" \cite{ABT18} and ``inverse problems" \cite{T05} in the literature. Generally, inverse problems make use of knowledge of $\mathcal{G}(x;\theta)\approx data(\theta)$ for a set of values of $\theta$ to infer the unknown parameter $x$, where $\mathcal{G}$ is the forward map, and $data$ is the measurement. Both $\mathcal{G}$ and $data$ are functions of $\theta$, the design variable, which is located in a design space $\Omega$. {This variable can define experimental choices, for instance the spatial and temporal coordinates of a source and/or a detector~\cite{brunton}}. We do not assume the measurements $data(\theta)$ to be precise; they may contain measurement error. The problem under study is inspired by PDE-based inverse problems where it is typical to assume that $x$ is a function ---an infinite dimensional object. In this paper, we  focus our attention on the case in which $x$ is a finite-dimensional object in  Euclidean space.

The need to collect informative data economically gives rise to the area of optimal experimental design (OED)~\cite{P06}, which seeks experimental setups that optimize certain statistical criteria. 
Mathematically, the OED problem assigns weights to each possible value of $\theta \in \Omega$ to optimize some statistical criterion. 
When the design space $\Omega$ is finite in size ($|\Omega|=m$), that is,  $\Omega=\{\theta_i\}_{i=1}^m$, the OED weights to be optimized can be gathered in a (finite-dimensional) vector $w=(w_1, w_2, \dotsc, w_m)^\top$, with $\sum_i w_i = 1$ and $w\geq 0$. (The latter vector inequality holds component-wise.)
In many experiments, however, the design space $\Omega$ is continuously indexed and often a subset of $\mathbb{R}^d$. For example, in the tomography problem, sensors can be located anywhere in the physical domain. In such situations, $|\Omega|$ is formally replaced by cardinality of the continuum, and OED weights $w(\theta)$ is continuously indexed. For continuously indexed set $\Omega$, one needs to deploy the notation of Borel sets, and for a fixed $\Omega$, we fix a preset Borel measure throughout without explicitly using it.

A na\"ive strategy to handle the continuous-valued design space is to discretize  $\Omega$ and represent it by $m$ values $\{\theta_i\}_{i=1}^m$, defined a priori, thus reducing the infinite-in-size problem to the finite classical setting. In~\cite{A69}, Atwood showed that if the target to be reconstructed is finite-dimensional, then there exists an optimal design that is supported only on finitely many points, giving this natural approach a justification. Caution is needed in interpreting this theorem: It is unclear that a brute-force discretization can fully capture this finite, measure-zero set. 
One such example is presented in Figure~\ref{fig: motivation}, where we showcase the data structure of the optical tomography problem in the strong photon regime~\cite{AS09,R10}. 
Lasers are placed on the boundary of an unknown medium and the scattered light intensity is detected also around the boundary. The design variable here is $\theta=[\theta^1,\theta^2]$ where $\theta^1$ denotes the incoming laser location and incident angle and $\theta^2$ denotes that of the detector. We concatenate the locations and the angles attributes into a 1D thread of indices for both $\theta^1$ and $\theta^2$. Both variables are continuous-valued. Figure~\ref{fig: motivation} shows that for different incoming $\theta^1$, there are only a few locations for the detectors that can capture the scattered field, as shown in the very thin strip of bright yellow. 
This observation implies that useful information is contained in a subdomain whose measure is almost zero. An arbitrary a-priori discretization would not be able to identify the informative sensor placement.

\begin{figure}
\centering
\includegraphics[scale=0.4]{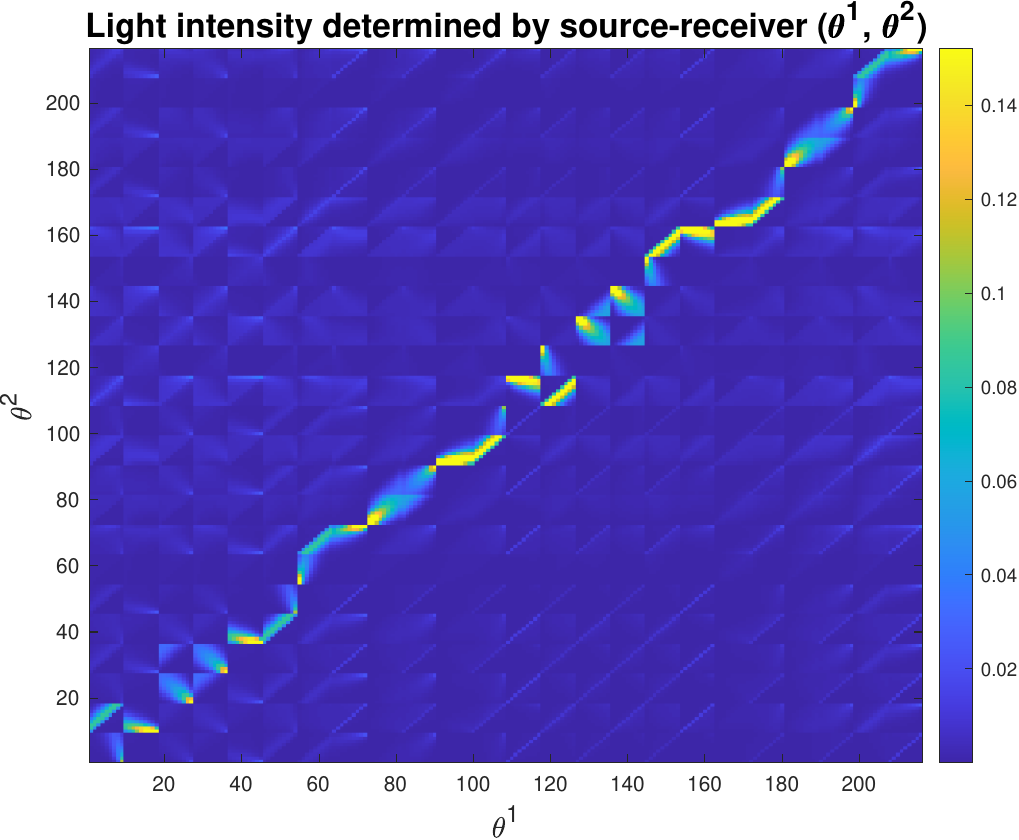}
\caption{Light intensity at receivers (denoted $\theta^2$) for light sources injected at different locations ($\theta^1$). The index $\theta^1$ or $\theta^2$ accounts for both the experiment location as well as the angles. The underlying model is the radiative transfer equation~\cite{chandrasekhar2013radiative}. It is clear that for each light source location, only limited sensor locations can receive non-trivial information.}
\label{fig: motivation}
\end{figure}

We would like to be able to capture an optimal set of observations without specifying candidate values of the observation variables in advance. One possible approach is encapsulated in the following question:
\smallskip
\begin{quote}
\emph{How do we solve the OED problem over a continuous design space?}
\end{quote}
\smallskip

Studying the OED problem from the perspective of continuous design space is not new; in \cite{KW59}, the OED problem was formulated as an optimization problem over the probability measure space. At that time, it was not fully understood how to metricize probability measure space and perform optimization over it. Thus the approach was quickly translated into a similar optimization problem constructed over the finite design space with $|\Omega|<\infty$. Now, recent techniques from optimal transport \cite{FG15} and Wasserstein gradient flow~\cite{AGS05} are available and  have yet to be integrated with experimental design. The main contribution of this paper is to take this step, using gradient flow as our main algorithmic tool.

There is an extensive literature on OED; we summarize relevant works in Section~\ref{subsec: related works}. We summarize our contributions and outline the remainder of the paper in Section~\ref{subsec: our contributions}.

\subsection{Related works}
\label{subsec: related works}
OED has been studied in the literature of statistics, applied mathematics, and machine learning, as well as in certain scientific domains. In the earliest stages of OED development, \cite{K74} gave rigorous justifications for various design criteria. 
Computationally, most early OED methods focus on discrete and combinatorial algorithms, manipulating the weights on points in finite design spaces.

Progress in computing OED has benefited from advances in several upstream research areas. Improved scientific computing capabilities have accelerated the evaluation of forward models; advances in optimization solvers~\cite{HX14,CABK19} have made OED computations faster and more scalable; and theoretical insights from randomized linear algebra~\cite{APSG16,WCG23-1} provide effective warm starts for OED. Within the OED framework, goal-oriented approaches~\cite{AS18,WCG23-2} exploit problem structure, while greedy strategies that is closely related to active learning in machine learning seek sequential sampling policies~\cite{CGJ96,JM21}.

To focus our discussion, we highlight existing works that specifically address optimal design formulated over the probability measure space. Neitzel et al. \cite{NPVW19} formulate sparse sensor placement as an optimization problem over the space of positive Borel measures. This formulation mitigates the non-convexity and combinatorial complexity associated with fixed design points. Their algorithmic approach employs conditional gradient methods, drawing on classical ideas from sequential and vertex exchange methods \cite{JD75}. Similarly, \cite{Molchanov04} frames OED in the measure space, with a particular focus on Poisson-type designs, and employs a gradient-descent method for the optimization. While these measure-based formulations broaden the OED modeling, their computations remain tied to finite support and do not directly address design over continuous measure spaces. In  contrast, our gradient-flow approach is grounded in optimal transport theory.

Other  works that leverage measure-based ideas for Bayesian design, making the estimation of design criteria tractable, include those of Huan \cite{huan2015numerical} and Koval et al. \cite{koval2024tractable}. 
These approaches construct measure-transport maps to approximate the joint distributions arising in nonlinear inverse problems. While transport maps facilitate the Bayesian computation, the OED problem itself is not formulated as an optimization over probability measures.

Another line of research aims to enhance computational efficiency by relaxing the optimality condition. 
In this regard, sampling and sketching techniques are crucial, especially in works that adopt the perspective of numerical linear algebra.
A comprehensive introduction to these flexible randomized methods can be found in the reviews \cite{mahoney2011randomized,Martinsson_Tropp_2020,hellmuth2026dataselectioninterfacepdebased}.

\subsection{Our contributions}
\label{subsec: our contributions}

The main contribution of this paper is  a computational framework for solving OED over a continuous design space. Inspired by recent developments in optimal transport, we define a gradient flow scheme for optimizing a smooth probability distribution driven by the OED objective on the Wasserstein metric. We  use Monte-Carlo particle approximation to translate the continuum flow of probability measure into gradient-descent flow for the finite set of sample particles, whose evolution captures the dynamics of the underlying infinite-dimensional flow. This evolution can be characterized by a coupled system of ordinary differential equations (ODE). We investigate theoretical aspects of the proposed technique, including convexity, criticality conditions, computation of Fr\'{e}chet derivatives, and convergence error with respect to key hyperparameters in the particle gradient flow algorithm. Finally, we apply our approach to two problems, with one from medical imaging: the linearized electrical impedance tomography (EIT) and the other related to inverse Darcy flow.
The experimental design produced by our algorithm  provides informative guidance for sensor placement.

The remainder of the paper is organized as follows. 
We prepare for the technical background on the OED problem and the gradient flow technique in Section~\ref{sec: problem}. 
In Section~\ref{sec: methodology}, we explain gradient flow for optimal design on continuous space, and introduce the particle gradient flow algorithm, Algorithm~\ref{alg: particle flow}. In Section~\ref{sec: theory}, we provide the theoretical properties of continuous OED optimization, including convergence guarantees for Algorithm\ref{alg: particle flow}. Finally, we test \ref{alg: particle flow} on two examples. The first is the linearized EIT inverse problem, whose numerical set-up and design performance are explained in Section~\ref{sec: setup} and \ref{sec: numerics}, respectively. 
The second test on 1D Darcy flow is described in Section~\ref{sec: new numerics}

%% file: sections/2-problem.tex
\section{Preliminaries and toolkits}
\label{sec: problem}

We present here the OED problem in its conventional discrete setting (Section~\ref{sec:oed}).
We then describe Wasserstein gradient flow, a fundamental tool that enables extension of OED to the continuous sampling space (Section~\ref{sec:wass}).

\subsection{Optimal experimental design} \label{sec:oed}
To introduce the classical OED setup, we consider the linear regression model:
\begin{equation}
\label{eqn: linear model}
\textbf{data}= \mA \vx^* + \eps.
\end{equation}
The number of measurements is $m \in \mathbb{N}^+$, with  observations collected in the entries of vector $\textbf{data}\in \Real^m.$
The (linear) forward observation map is encoded in the matrix $\mA \in \Real^{m \times d} ~(m \gg d)$, with random noise contributions in the vector $\eps \in \Real^{m}$.
We wish to infer the parameters $\vx^*\in \Real^d.$ 

For the vastly overdetermined system \eqref{eqn: linear model}, an estimate of $\vx^*$ can be obtained without requiring access to  the full map $\mA$. The aim of OED is to identify a combination of  measurements that enables accurate yet economical recovery. 
Specifically, since each row of the system \eqref{eqn: linear model} represents an {\em experiment},  we seek a vector $w = (w_1,w_2,\dotsc,w_m)^\top$ whose components represent the {\em weights} that we assign to each experiment, that solves the following problem:
\begin{equation}
\min_w \, F[w]   \;\; \mbox{subject to} \;\; w \geq 0, \;  \sum_{i=1}^m w_i = 1.
\end{equation}
The function $F: \Real^m \to \Real$ represents certain design criterion, with smaller objective values of $F$ implying better design.

Many statistical criteria have been  proposed for OED. 
We present the two most commonly used standards \cite{K74}, denote by the letters ``A" (for ``average") and ``D" (for ``determinant").

It is well known that the  optimal inference result for \eqref{eqn: linear model} {(under the $\ell_2$ metric)} that makes use of all data is 
\[
\hat{\vx} = \big(\mA^\top \mA \big)^{-1} \mA^\top \textbf{data}.
\]
When the noise vector is assumed to follow an i.i.d. Gaussian distribution, that is, 
$\eps\sim\mathcal{N}(0,\sigma^2 \id_m)$, the variance matrix of the solution $\hat{\vx}$ is
\begin{eqnarray*}
\text{var}(\hat{\vx}) & = & \mathbb{E} \big[(\hat{\vx} - \vx^*)(\hat{\vx} - \vx^*)^\top \big] \\
& =  &\mathbb{E}\big[ \big(\mA^\top \mA \big)^{-1} \mA^\top \eps \, \eps^\top \mA \big(\mA^\top \mA \big)^{-1} \big] \\
& = & \big(\mA^\top \mA \big)^{-1} \mA^\top \mathbb{E} [\eps \, \eps^\top] \mA \big(\mA^\top \mA \big)^{-1} \\
& = & \sigma^2 (\mA^\top \mA)^{-1}.
\end{eqnarray*}
The above calculation shows that the inference ``uncertainty" depends on the property of the data matrix through the term $(\mA^\top \mA)^{-1}$. We note at the outset that this formulation holds only under the assumption that $\mA^\top \mA$ is invertible, that is, the columns of $\mA$ span the full $d$-dimensional space. In the setting where $m\gg d$, this assumption is usually satisfied.

Since smaller variances indicate lower levels of uncertainty and a more accurate reconstruction, we can modify this variance matrix by weighting the experiments using the weights $w$, leading to the following weighted inverse variance \cite[Section~7.5]{BoyV03}:
\begin{equation}\label{eqn:A_T_A_w}
\mA^\top \mA [w] := \sum_{k=1}^m w_k \,\mA(k,:)^\top \mA(k,:).
\end{equation}
The OED problem chooses $w$  to minimize a scalar function of the inverse of this weighted variance matrix. The A- and D- design criteria are defined as follows:

\begin{subequations}
\begin{eqnarray}
\label{eqn: discrete A-opt}
    \text{A-optimal:} \quad \min F^A[w] & \equiv & \tr\big(\mA^\top \mA [w] \big)^{-1}, \\
    \label{eqn: discrete D-opt}
    \text{D-optimal:} \quad \max F^D[w] & \equiv & \log\big(\det\big( \mA^\top \mA [w ]\big)\big).
\end{eqnarray}
\end{subequations}

\begin{remark}
In some scenarios, one would further like the weight vector $w$ to be sparse, so that only a small number of experiments are chosen. One well studied approach deploys the classical concept of leverage score~\cite{DMMW12,MMY14}. 
\end{remark}

\subsection{Wasserstein gradient flow} \label{sec:wass}

We describe here the basics of gradient flow \cite{AGS05} and related methods.
Analogous to gradient descent in Euclidean space, gradient flow optimizes a probability measure objective by defining a flow in the variable space based on a gradient of the objective function. Proper metricization of the space is  a critical issue.
In this regard, we leverage significant advances in optimal transport \cite{S15,FG21} and Wasserstein gradient flow \cite{JKO98,AGS05}, reviewed below.
 
We require the class of probability measures $\rho$ to have bounded second moments, that is,
\begin{equation}
\label{eqn: Pr2}
\pr_2 (\Omega) =\left\{ \rho \mid \rho \text{ is a Borel probability measure and }\int_\Omega |\theta|^2\dd\rho(\theta) < \infty\right\}\,.
\end{equation}
Note that the probability distribution $\rho$ is not necessarily absolute continuous with respect to the Borel measure assigned to $\Omega$. 
Dirac delta functions can be used, enabling practical computations. It is natural to equip the $\pr_2$ space \eqref{eqn: Pr2} with the Wasserstein-2 metric to measure the distances between probability distributions.  
For this purpose, we define the joint probability measure $\gamma: \Omega \times \Omega \to \Real$ and  the set $\Gamma(\mu, \nu)$ to be the space of joint probability measures whose first and second  marginals are  $\mu \in \pr_2(\Omega)$ and $\nu \in \pr_2(\Omega)$, respectively. 

\begin{definition}
Given the domain $\Omega,$ the Wasserstein-$2$ distance between two probability measures $\mu, \nu \in \pr_2(\Omega)$ is defined as 
\begin{equation}
\label{eqn: W2 metric}
W_2(\mu, \nu) : = \inf_{\gamma \in \Gamma(\mu, \nu)}\left( \int_{\Omega \times \Omega} \|x-y\|^2 \dd \gamma(x,y) \right)^{1/2}, 
\end{equation}
where the norm in the integrand is the Euclidean $\ell_2$-distance in $\Omega$.
\end{definition}

Over the space $\pr_2(\Omega)$, consider an objective functional $F: \pr_2(\Omega) \to \Real$ with the optimization problem:
\[
\min_{\rho\in\pr_2(\Omega)} \, F[\rho]\,.
\]
Analogous to gradient descent in Euclidean space, where the ODE $\dot{x}=-\nabla_xf$ drives the iterates toward minimizers of $f$, a natural solver in this setting is the gradient flow. Due to the nonlinear structure of $\pr_2(\Omega)$ and the geometry induced by the Wasserstein metric, the notion of gradient takes a nonstandard form. The descent dynamics are given by the Wasserstein gradient flow of $F$ (see \cite{S15}, Chapter 8.2, and the celebrated JKO scheme \cite{JKO98}):
\begin{equation}
\label{eqn: W2 flow}
\partial_t \rho = - \nabla _{W_2} F[\rho] = \nabla_\theta\cdot  \left(\rho \, \nabla_\theta \frac{\delta F [\rho]}{\delta \rho} \right).
\end{equation}
If $F$ is convex along $W_2$-geodesics, then the long-time evolution of \eqref{eqn: W2 flow} converges to a minimizer of $F$, whenever such a minimizer exists.

%% file: sections/3-method.tex
\section{Optimal design via gradient flow}
\label{sec: methodology}

In this section, we start by defining the optimal design problem in continuous space, defining continuous analogs of the two objective functions  in \eqref{eqn: discrete A-opt} and \eqref{eqn: discrete D-opt} in the probability measure space, and obtaining expressions for the gradients of these functionals. 
Next, we define a particle approximation to simulate this gradient flow, as summarized in \ref{alg: particle flow}, so to optimize the objective functionals.

\subsection{Continuous optimal design}
In the continuous setting, the $m \times d$ matrix $\mA$ of \eqref{eqn: linear model} is replaced by an object that we call the ``continuous experiment/data matrix" with $d$ columns but ``row" space indexed by $\theta\in\Omega$, so that  $\mA(\theta,:)$ is a real row vector with $d$ elements. We also define a ``continuous" counterpart of the weighted matrix product of \eqref{eqn:A_T_A_w}:
\begin{equation}
\label{eqn: sampled ATA}
\mA^\top \mA [\rho] = \int_\Omega \mA(\theta,:)^\top \mA(\theta,:)\dd \rho(\theta) \in \mathbb{R}^{d \times d}. 
\end{equation}
Accordingly, following the A- and D-optimal discrete design in~\eqref{eqn: discrete A-opt} and \eqref{eqn: discrete D-opt}, we arrive at the corresponding criteria in the continuous context: 
\begin{eqnarray}
\label{eqn: cont A-opt}
\text{A-optimal:} \quad \rho^* & := & \arg\min_{\rho \in \pr_2(\Omega)} F^A[\rho] \equiv \tr\big(\mA^\top \mA [\rho] \big)^{-1}, \\
\label{eqn: cont D-opt}
\text{D-optimal:} \quad \rho^* & :=& \arg\max_{\rho \in \pr_2(\Omega)} F^D[\rho] \equiv \log\big(\det\big( \mA^\top \mA [\rho]\big)\big).
\end{eqnarray}
These definitions require the continuous experiment/data matrix $\mA$ to have full column rank $d$. This is a weak assumption because the continuous design space $\Omega$ would generally not have uniform dependencies among the components of $\mA(\theta,:)$ for all $\theta \in \Omega$.

To apply  Wasserstein gradient flow~\eqref{eqn: W2 flow}\footnote{Note that D-optimal design formulates a maximization problem \eqref{eqn: cont D-opt}. Contrary to A-optimal design, the associated gradient-flow follows in the ascending direction of the gradient. Consequently, the sign in \eqref{eqn: W2 flow} should be flipped.}, we need to prepare first variation of the OED objectives. 
These are defined in the following result.
\begin{proposition}
\label{prop:gradient_flow_cont}
Assume that the continuous experiment/data matrix $\mA$ defined in this section is full rank. The first variation for~\eqref{eqn: cont A-opt} and~\eqref{eqn: cont D-opt} are
\begin{eqnarray}
\label{eqn: A derivative}
&& \frac{\delta F^A[\rho]}{\delta \rho}(\theta) = -\mA(\theta,:) \big(\mA^\top \mA [\rho] \big)^{-2} \mA(\theta,:)^\top\,, \\
\label{eqn: D derivative}
&& \frac{\delta F^D[\rho]}{\delta \rho}(\theta) = \mA(\theta,:) \big(\mA^\top \mA [\rho] \big)^{-1} \mA(\theta,:)^\top\,. 
\end{eqnarray}
\end{proposition}
\begin{proof}
Given $\rho\in\pr_2(\Omega)$, the first variation of $F$ at $\rho$ is a function $\frac{\delta F}{\delta\rho}:\Omega\to\Real$ such that, for any signed measure $\eta$ with zero total mass (that is, $\eta(\Omega)=0$), and for sufficiently small $\epsilon$ such that $\rho+\epsilon\eta\in\pr_2(\Omega)$, one has
\begin{equation}
\label{eqn: Frechet derivative}
F[\rho + \epsilon \eta] - F[\rho] = \epsilon 
\int_\Omega \frac{\delta F (\theta) }{\delta\rho }\dd \eta(\theta) +o(\epsilon). 
    \end{equation}
For the A-optimal criterion \eqref{eqn: cont A-opt}, we have
\begin{eqnarray*}
        F^A[\rho + \epsilon \eta] - F^A[\rho] &=& \displaystyle \tr\big(\mA^\top \mA [\rho + \epsilon \eta] \big)^{-1} -  \tr\big(\mA^\top \mA [\rho] \big)^{-1} \\
  & = & \tr\Big( \big(\mA^\top \mA [\rho] +\mA^\top \mA [\epsilon \eta]\big)^{-1} - \big( \mA^\top \mA [\rho] \big)^{-1} \Big)\\
  & = & \epsilon\tr\Big( -\big(\mA^\top \mA [\rho]\big)^{-1} \big(\mA^\top \mA [\eta]\big)\big(\mA^\top \mA [\rho]\big)^{-1}\Big) +O(\epsilon^2).
\end{eqnarray*}
With more finite-dimensional matrix calculation, and compare with~\eqref{eqn: Frechet derivative}, we have
    \begin{eqnarray}
   \nonumber
    \int_\Omega \frac{\delta F^A }{\delta \rho} (\theta)  \dd \eta (\theta) &=&    \tr\Big( -\big(\mA^\top \mA [\rho]\big)^{-1} \big(\mA^\top \mA [\eta]\big)\big(\mA^\top \mA [\rho]\big)^{-1}\Big)\\   
    \nonumber
    &= & \tr\Big( -\big(\mA^\top \mA [\rho]\big)^{-2} \mA^\top \mA [\eta]\Big)  \\
    \nonumber
    & = & \tr\Big(-\big(\mA^\top \mA [\rho]\big)^{-2} \int_\Omega \mA(\theta,:)^\top \mA(\theta,:)   \dd \eta (\theta)  \Big) \\
     \label{eqn: A difference}
  & = & \int_\Omega \tr\Big(-\big(\mA^\top \mA [\rho]\big)^{-2} \mA(\theta,:)^\top \mA(\theta,:)\Big)\dd\eta (\theta)  \\
  \nonumber
 & = & \int_\Omega {-\mA(\theta,:) \big(\mA^\top \mA [\rho]\big)^{-2}  \mA(\theta,:)^\top}  \dd \eta (\theta) \,.  
\end{eqnarray}
The above results make use of the fact that the trace of matrix products is commutative (lines 2 and 5) and also that the trace and integration operations are interchangeable (line 4). Since this holds true for all $\eta$, we obtain the first variation of A-design:
\[
\frac{\delta F^A  }{\delta \rho} (\theta)  = -\mA(\theta,:) \big(\mA^\top \mA [\rho]\big)^{-2}  \mA(\theta,:)^\top.
\]

A similar derivation holds for the D-optimal objective \eqref{eqn: cont D-opt}. We consider 
\[
F^D[\rho + \epsilon\eta] - F^D[\rho] 
= \log\Big(\det\big( \mA^\top \mA [\rho + \epsilon\eta]\big)\big)  - \log\big(\det\big( \mA^\top \mA [\rho]\big)\Big).
\]
The linear approximation to the first term on the right-hand side is 
\begin{eqnarray*}
&& \log\Big(\det\big( \mA^\top \mA [\rho + \epsilon\eta]\big)\Big)\\
& = &
\log\Big(\det\big( \mA^\top \mA [\rho]\big)\Big) + \epsilon\, \tr\Big((\mA^\top \mA [\rho])^{-1}\mA^\top \mA [\eta]\Big)+O(\epsilon^2)\,,
\end{eqnarray*}
where we used Jacobi's formula for derivative of matrix determinant, and the $O(\epsilon^2)$ term depends only on spectral information of $\mA$ and is independent of $\eta$. Comparing with~\eqref{eqn: Frechet derivative}, we then have the linear difference term for $F^D$:
\begin{eqnarray*}
    F^D[\rho + \epsilon\eta] - F^D[\rho] & = &\epsilon
\tr\Big(( \mA^\top \mA [\rho])^{-1} \mA^\top \mA [\eta]\Big) +O(\epsilon^2)\\
&= & \epsilon\int_\Omega \mA(\theta,:)^\top (\mA^\top \mA [\rho])^{-1}  \mA(\theta,:) \dd \eta (\theta) +O(\epsilon^2)\,.\\ 
\end{eqnarray*}
Compare it with~\eqref{eqn: Frechet derivative}, we have:
\[
\frac{\delta F^D  }{\delta \rho}(\theta) = \mA(\theta,:) \big(\mA^\top \mA [\rho]\big)^{-1}  \mA(\theta,:)^\top.
\]
\end{proof}

\subsection{Particle gradient flow}
\label{subsec: particle gradient flow}
Proposition~\ref{prop:gradient_flow_cont} in combination with \eqref{eqn: W2 flow} defines the gradient flow for finding the OED probability measure over the design space $\Omega$. Classical techniques for solving this PDE formulation involve discretizing $\Omega$ and tracing the evolution of $\rho$ on the resulting mesh. 
This strategy presents a  computational challenge: The size of the mesh (or equivalently, the degrees of freedom required to represent $\rho$ in the discrete setting) grows exponentially with the dimension of the design space. 
The computational complexity required to implement this strategy would exceed the experiment budget, in terms of the optimized weighting object size and total measurements. 

One advantage of employing the Wasserstein gradient flow is its close relationship to a particle ODE interpretation \cite{CHPW15,BCDP15,CB18}. 
We can use Monte Carlo to represent the probability measure $\rho$ by a particle samples on the design space $ \Omega$. 
This simulation translates the PDE into a coupled ODE system on the sample vector $\theta \in \Omega$. Following the descending trajectory of~\eqref{eqn: W2 flow}, when $\rho$ is known, the characteristic of this PDE is 
\begin{equation}
\label{eqn: particle flow}
    \frac{\dd \theta }{\dd t} = - \nabla_\theta \frac{\delta F[\rho]}{\delta \rho}(\theta).
\end{equation}
In computation, the distribution $\rho$ is unknown, but we can use an empirical measure for the approximation to $\rho$. Given a fixed number of particles $N \in \mathbb{N}^+,$ we consider a set of particles $\{ \theta_i \}_{i=1}^N \subset \Omega.$ The estimated probability $\rho$ is the average of Dirac-delta measures at selected particles $\rho_N$, that is, 
\begin{equation}
\label{eqn: empirical measure}
 \rho \approx \rho_N = \frac{1}{N} \sum_{i=1}^N  \delta_{\theta_i} \in \pr_2(\Omega).
\end{equation}
The distribution $\rho_N$  approximates the continuous range of $\Omega$ better than a costly setup in which  discretization is predefined in Section~\ref{subsec: particle gradient flow}. By inserting the empirical measure~\eqref{eqn: empirical measure} into~\eqref{eqn: particle flow}, and employing forward-Euler time integration, we arrive at the following particle gradient flow method.

\begin{algorithm}[!htb]
\caption{Particle gradient flow}

 \textbf{Input:} 
Number of particles $N$;
number of iterations $T$; time step $\dd t$; \\
initial particles $\theta_1^0, \dotsc, \theta_N^0 \subset \Omega$ and starting measure $\rho^0_N = \frac{1}{N} \sum_{i=1}^N \delta_{\theta^0_i}  $\\
 \textbf{Output:} probability measure $\rho \in \pr_2 (\Omega)$
 \vspace{1mm}
 \begin{algorithmic}[1]
\For{$t = 1, \dotsc, T$}
 \For{$i = 1, \dotsc, N$}
\State $\displaystyle \theta^t_i \gets \theta^{t-1}_i - \dd t \, \nabla_{\theta} \frac{\delta F[\rho_N^{t-1}]}{\delta \rho} (\theta_i^{t-1})$ \Comment{update particles (descent)}
\EndFor
 \State $\displaystyle \rho_N^t \gets \sum_{i=1}^N \frac{1}{N} \delta_{\theta_i^{t}}$ \Comment{update probability measure}
 \EndFor
\State \Return $\rho \gets \rho_N^T$
\end{algorithmic}
\label{alg: particle flow}
\end{algorithm}

\begin{remark}
Line 3 of Algorithm~\ref{alg: particle flow} is the descent update formula for  minimization of $F$. 
For maximization, as in  the D-optimal design \eqref{eqn: cont D-opt}, we switch the minus sign to plus to obtain ascent.
\end{remark}

\begin{remark}
We stress the difference between this algorithm and the classical OED pursuit methods, such as the Fedorov method in~\cite{F13}. Our algorithm is the particle Monte Carlo method that implements the OED gradient flow formulations \eqref{eqn: A velocity}-\eqref{eqn: D velocity}. As in other Monte Carlo method for gradient flow, each particle carries the same weight --- but the {\em locations} of the particles are being updated. By contrast, the conventional methods update only the {\em weights}. \end{remark}

Algorithm~\ref{alg: particle flow} requires calculation of the particle velocity forms $\nabla_\theta \frac{\delta F[\rho_N]}{\delta \rho}$ for A- and D-optimal design criteria \eqref{eqn: cont A-opt}, \eqref{eqn: cont D-opt}. 
Details of this computation are shown in the next result.

\begin{proposition}\label{lemma: velocity}
Fix a set of particles $\{\theta_i \}_{i=1}^N \subset \Omega$ and consider the empirical measure \eqref{eqn: empirical measure}. 
The flow field of a sample particle $\theta \in \Omega$   under the A-optimal objective \eqref{eqn: cont A-opt} is
\begin{equation}
\label{eqn: A velocity}
\nabla_{\theta} \frac{\delta F^A[\rho_N]}{\delta \rho} (\theta) =-2\, \nabla_\theta \mA(\theta,:) \Big( \frac{1}{N} \sum_{i=1}^N \mA(\theta_i,:)^\top \mA(\theta_i,:) \Big)^{-2}\mA(\theta,:)^\top.
\end{equation}
For the D-optimal objective \eqref{eqn: cont D-opt} the flow field is
\begin{equation}
\label{eqn: D velocity}
\nabla_{\theta} \frac{\delta F^D[\rho_N]}{\delta \rho} (\theta) =2\, \nabla_\theta \mA(\theta, :) \Big( \frac{1}{N} \sum_{i=1}^N \mA(\theta_i,:)^\top \mA(\theta_i, :)\Big)^{-1}\mA(\theta,:)^\top.
\end{equation}
\end{proposition}

\begin{proof}

For the empirical measure \eqref{eqn: empirical measure},  the sampled target \eqref{eqn: sampled ATA} is
\[
\mA^\top \mA[\rho_N] = \frac{1}{N}\sum_{i=1}^N \mA(\theta_i,:)^\top \mA(\theta_i,:).
\]
We substitute this term into the Fr\'{e}chet derivatives of \eqref{eqn: A derivative} and \eqref{eqn: D derivative}. 
For any particle $\theta \in \Omega,$ we obtain 
\begin{eqnarray*}
\frac{\delta F^A[\rho_N]}{\delta \rho}(\theta) &=& -\mA(\theta,:) \Big( \frac{1}{N} \sum_{i=1}^N \mA(\theta_i,:)^\top \mA(\theta_i,:) \Big)^{-2}\mA(\theta,:)^\top, \\
\frac{\delta F^D[\rho_N]}{\delta \rho}(\theta) &=& \mA(\theta,:) \Big( \frac{1}{N} \sum_{i=1}^N \mA(\theta_i,:)^\top \mA(\theta_i,:) \Big)^{-1}\mA(\theta,:)^\top.
\end{eqnarray*}
Note that in the computation for both derivatives, the middle matrix terms are already evaluated at fixed values $\{\theta_i\}_{i=1}^N$ and thus are independent of $\theta$. When taking the gradient with respect to $\theta$, only the      {terms $\mA(\theta,:)$} contribute, and we arrive at
\begin{eqnarray*}
\nabla_\theta \frac{\delta F^A[\rho_N]}{\delta \rho}(\theta) 
& = & - 2\, \nabla_\theta \mA(\theta,:) \Big( \frac{1}{N} \sum_{i=1}^N \mA(\theta_i,:)^\top \mA(\theta_i,:)\Big)^{-2}\mA(\theta,:)^\top, \\
\nabla_\theta \frac{\delta F^D[\rho_N]}{\delta \rho}(\theta)
& = & 2\, \nabla_\theta \mA(\theta,:) \Big( \frac{1}{N} \sum_{i=1}^N \mA(\theta_i,:)^\top \mA(\theta_i,:)\Big)^{-1}\mA(\theta,:)^\top.
\end{eqnarray*} 
\end{proof}

%% file: sections/4-theory.tex
\section{Theoretical guarantees}
\label{sec: theory}
We provide some theoretical properties regarding the continuous OED and the proposed Algorithm~\ref{alg: particle flow}. 
We study the first-order critical condition  in Section~\ref{subsec: criticality} and the convexity of the OED objective functionals Section~\ref{subsec: convexity}. In Section~\ref{subsec: simulation error}, we discuss the convergence of Algorithm~\ref{alg: particle flow}. 

\subsection{First-order critical condition}
\label{subsec: criticality}
We first specify the stationary point condition under the $W_2$ metric.
\begin{lemma}
\label{lemma: criticality}
The distribution $\rho^\ast \in \pr_2(\Omega)$ is a stationary solution to $\min_{\rho \in \pr_2(\Omega)}$ or $\max_{\rho \in \pr_2(\Omega)} F[\rho]$  --- that is,   $\partial_t\rho^\ast = 0$ --- if it satisfies the following first-order critical condition:
\begin{equation}
\label{eqn: criticality}
\nabla_\theta \frac{\delta F[\rho^*]}{\delta \rho}(\theta) = 0, \quad \forall ~\theta \in \text{\rm supp}(\rho^*).
\end{equation}
\end{lemma}

\begin{proof}
We present the proof for the minimization problem. (The proof for maximization is similar.)
In the Wasserstein flow of $\rho \in \pr_2(\Omega)$, the differential of the objective $F$ is 
\begin{equation}
\label{eqn: dF}
\frac{\dd F[\rho]}{\dd t} = \int_{\Omega} \frac{\delta F[\rho]}{\delta \rho} (\theta) \, \partial_t \rho(\theta) \,  \dd \theta.
\end{equation}
When we substitute for  $\partial_t \rho $ from \eqref{eqn: W2 flow}, we obtain
\begin{equation}
\label{eqn: dF 2}
\frac{\dd F[\rho]}{\dd t} = \int_\Omega \frac{\delta F [\rho]}{\delta \rho}(\theta) \, \nabla_{\theta} \, \cdot \left(\rho \nabla_\theta \frac{\delta F[\rho]}{\delta \rho} \right) \dd \theta
= - \int_\Omega \rho(\theta) \left \vert \nabla_\theta \frac{\delta F[\rho]}{\delta \rho} \right\vert^2 \, \dd \theta \leq 0.  
\end{equation}
The last derivation is from the Green's identity and the assumption that the velocity term $\nabla_\theta \frac{\delta F[\rho]}{\delta \rho}$ vanishes on the boundary $\partial \Omega.$

First-order criticality conditions for $\rho^*$ are $\frac{\dd F [\rho^*] }{\dd t} = 0$. Hence equation \eqref{eqn: dF 2} implies that the critical condition  \eqref{eqn: criticality} is required for the integrand in \eqref{eqn: dF 2} to be zero everywhere in $\Omega.$
\end{proof}

In the $W_2$ descent flow \eqref{eqn: W2 flow}, the objective $F$ keeps decreasing until $\rho$ achieves stationarity.  
(A similar claim applies to ascent in the maximization case.) The critical condition \eqref{eqn: criticality} does not give the explicit stationary measure $\rho^*$ except in special cases, one of which we present now.

\begin{proposition}
\label{ex: critical point}
Suppose that in the    {continuous experiment/data matrix $\mA$ there are a set of values $\{\theta_i \}_{i=1}^d \subset \Omega$ such that the vectors $\{\mA(\theta_i,:) \}_{i=1}^d$ are orthogonal, and in addition that}
\begin{equation}\label{assum:perp}
\nabla_\theta \mA(\theta_i,:)\perp\mA(\theta_i,:)\,,\quad  \forall i \in [d].
\end{equation}
Then the following form satisfies the A- \eqref{eqn: cont A-opt} and D-optimal \eqref{eqn: cont D-opt} design criteria:
    \begin{equation}
    \label{eqn: critical}
    \rho^{*} = \sum_{i=1}^d \alpha_i \delta_{\theta_i}, \quad \text{s.t.}~\sum_{i=1}^d \alpha_i = 1, ~~\alpha_i > 0, ~\forall i \in [d].
\end{equation}
\end{proposition}
\begin{proof}
By rescaling the orthogonal rows $\mA(\theta_i,:)$ for $i \in [d]$ we can define an orthogonal matrix $\mU \in \Real^{d \times d}$ with rows $\mU(i,:)$ defined by
\begin{equation}
\label{eqn: normalization}
\mU(i,:) = \frac{\mA(\theta_i,:)}{\|\mA(\theta_i,:) \|}, \quad \forall i \in [d].
\end{equation}
For the sampled target by $\rho^*$, we obtain from this formula and \eqref{eqn:A_T_A_w} that
\[
\mA^\top \mA [\rho^*] = \sum_{i=1}^d \alpha_i \mA(\theta_i,:)^\top \mA(\theta_i,:) = \mU^\top \mC \mU,
\]
where $\mC \in \Real^{d \times d}$ is a diagonal matrix with $i$th diagonal $\mC(i,i) = \alpha_i \| \mA(\theta_i, :) \|^2$.
We thus have by orthogonality of $\mU$ that 
\[
(\mA^\top \mA [\rho^*])^{-1} = \mU^\top \mC^{-1} \mU, \quad 
(\mA^\top \mA [\rho^*])^{-2} =  \mU^\top \mC^{-2} \mU.
\] 

For $\theta_i \in \text{supp} (\rho^*),$  the A-optimal derivative is
\begin{eqnarray}
\nonumber
\nabla_\theta \frac{\delta F^A [\rho^{*}] }{\delta \rho} (\theta_i) & =  & 2\nabla_\theta \mA(\theta_i,:) (\mA^\top \mA [\rho^{*}] )^{-2} \mA(\theta_i,:)^\top \\
\label{eqn:sh1}
& = & 2\nabla_\theta \mA(\theta_i,:) \mU^\top \mC^{-2} \mU \mA(\theta_i,:)^\top.
\end{eqnarray}
Since $\{\mA(\theta_i,:)\}_{i=1}^d$ have orthogonal rows, we have from \eqref{eqn: normalization} that
\[
\mC^{-2} \mU \mA(\theta_i,:)^\top 
= \| \mA(\theta_i,:) \|  \mC^{-2} \mU \mU(i,:)^\top 
= \frac{\| \mA(\theta_i,:) \| }{\mC(i,i)^2} \ve_i,
\]
{where $\ve_i$ is the $i$-the unit vector in $\Real^d$.}
By substituting into \eqref{eqn:sh1} and using \eqref{eqn: normalization} again, we obtain
\begin{eqnarray*}
\nabla_\theta \frac{\delta F^A [\rho^{*}] }{\delta \rho} (\theta_i) & 
  =  & 2\nabla_\theta \mA(\theta_i,:) \mU^\top \ve_i \frac{\| \mA(\theta_i,:)\|}{\mC(i,i)^2} \\
&  =  & 2\nabla_\theta \mA(\theta_i,:) \mU(i,:)^\top \frac{\| \mA(\theta_i,:)\|}{\mC(i,i)^2}\\
&  =  & 2\nabla_\theta \mA(\theta_i,:) \mA(\theta_i,:)^\top \frac{1}{\mC(i,i)^2}\\
&= &0,
\end{eqnarray*}
where the final equality follows from \eqref{assum:perp}.
A similar argument shows that $\nabla_\theta \frac{\delta F^D [\rho^{*}] }{\delta \rho} (\theta_i)=0$.

Since the gradient of the Fr\'{e}chet derivative is $0$ for all support points $\theta_i$, $\rho^{*}$ satisfies the first-order criticality condition \eqref{eqn: criticality}. 
\end{proof}

\subsection{Convexity of design objectives}
\label{subsec: convexity}

Another feature of the OED problems \eqref{eqn: cont A-opt}–\eqref{eqn: cont D-opt} is that they are linear convex/concave, in the sense that, for any $\rho_0$ and $\rho_1$, we have:
\[
F^A[(1-s)\rho_0+s\rho_1]\leq (1-s)F^A[\rho_0]+sF^A[\rho_1]\,,\quad F^D[(1-s)\rho_0+s\rho_1]\geq (1-s)F^D[\rho_0]+sF^D[\rho_1]\,.
\]
One way to show these properties is to evaluate the sign of the Hessian term. If for any given $\rho\in\pr_2(\Omega)$ and any admissible signed measure direction $\eta$ with $\eta(\Omega)=0$ to ensure $\rho+\epsilon\eta\in\pr_2(\Omega)$, one has the inequality $\left.\frac{\dd^2}{\dd\epsilon^2}F[\rho+\epsilon\eta]\right|_{\epsilon=0}\geq (\leq)0$, then $F$ is linear convexity(concavity). We will show this by directly computing the Hessian term in the following proposition.

\begin{proposition}
\label{prop: convexity}
For any $\rho$ that ensures invertibility of the matrix $\mA^\top \mA [\rho]$ defined in \eqref{eqn: sampled ATA}, the objective functionals for both A-optimal and D-optimal defined in~\eqref{eqn: cont A-opt} and~\eqref{eqn: cont D-opt} are second order directional differentiable, with the Hessian functionals $\hess~ F^A[\rho]$ and $\hess~ F^D[\rho]$ being semidefinite operators (positive and negative, respectively).
\end{proposition}

\begin{proof}
Fix a probability distribution $\rho \in \pr_2(\Omega),$ we will explicitly compute the two Hessian terms. For any given two signed measures $\eta_1, \eta_2: \Omega \to \Real$ ($\eta_1(\Omega)=\eta_2(\Omega)=0$) that live on the ambient space of $\pr_2$. The bilinear Hessian operator is computed by:
\begin{eqnarray}
\nonumber
&& \hess~ F[\rho](\eta_1, \eta_2) \\
\label{eqn: def_hess_original}
& = & \lim_{\epsilon\to0}\frac{(F[\rho+ \varepsilon\eta_1 + \varepsilon\eta_2] - F[\rho+ \varepsilon\eta_2]) - (F[\rho+ \varepsilon\eta_1 ] - F[\rho])}{\varepsilon^2}\\
\nonumber
& = & \lim_{\varepsilon\to0}\frac{1}{\varepsilon}\left( \int_\Omega \frac{\delta F[\rho + \varepsilon\eta_2 ]}{\delta \rho} \dd \eta_1 (\theta)   - \int_\Omega \frac{\delta F[\rho  ]}{\delta \rho} \dd \eta_1 (\theta) \right)\,.
\end{eqnarray}
The last derivation is from~\eqref{eqn: Frechet derivative}.

For positive definiteness, we need $\hess~ F[\rho](\eta, \eta)\geq0$ for all $\eta$. Since the sign is retained when passing to the limit, we will study the first-order expansion of
\begin{equation}
\label{eqn: hessian def}
\int_\Omega \frac{\delta F[\rho + \varepsilon\eta ]}{\delta \rho} \, \dd \eta (\theta)   - \int_\Omega \frac{\delta F[\rho  ]}{\delta \rho} \, \dd \eta (\theta) \geq0, \quad \forall~\varepsilon \in \Real.
\end{equation}

Regarding \eqref{eqn: sampled ATA}, we  define the following shorthand notation:
\begin{equation}
\label{eqn: short-hand notation}
\mM := \mA^\top \mA [\rho] \in \Real^{d \times d}, \quad 
\mD := \mA^\top \mA [\delta \rho] \in \Real^{d \times d}.
\end{equation}
(Note that the theorem assumes positive definiteness of $\mM$.)

We first study the A-optimal design objective \eqref{eqn: cont A-opt}. For the first term in \eqref{eqn: hessian def}, we deploy \eqref{eqn: A derivative} from \ref{prop:gradient_flow_cont} to obtain
\begin{equation*}
\int_\Omega \frac{\delta F^A[\rho + \varepsilon\eta ]}{\delta \rho} \, \dd \eta (\theta) 
= -\int_\Omega \mA(\theta,:) (\mA^\top \mA [\rho + \varepsilon \eta] )^{-2} \mA(\theta,:)^\top   \dd \eta (\theta) \,.
\end{equation*}
Since the integrand is a scalar, it is equivalent to its trace value. 
Thus the expression above is equal to
\begin{eqnarray*}
&&  -\int_\Omega \tr\big( \mA(\theta,:) (\mA^\top \mA [\rho + \varepsilon \eta] )^{-2} \mA(\theta,:)^\top\big)\dd \eta (\theta)\\
&=& - \int_\Omega\tr \big((\mA^\top \mA [\rho + \varepsilon \eta])^{-2} \mA(\theta,:)^\top   \mA(\theta,:)\big)\dd \eta (\theta)\\
&=&- \tr \Big(\int_\Omega (\mA^\top \mA [\rho + \varepsilon \eta])^{-2} \mA(\theta,:)^\top   \mA(\theta,:)\dd \eta (\theta)\Big)\\
&=& - \tr \Big( (\mA^\top \mA [\rho + \varepsilon \eta])^{-2}\int_\Omega \mA(\theta,:)^\top   \mA(\theta,:)\dd \eta (\theta) \Big)\\
&=& - \displaystyle \tr\big((\mA^\top \mA [\rho] + \varepsilon\mA^\top \mA [\eta] )^{-2} \mA^\top\mA [\eta]\big). 
\end{eqnarray*}
The second line follows from the cyclic property of matrix product in the trace operation.
The order of trace and integration can be switched as both are linear operators (third line). 
The last line follows from linear expansion of the $\mA^\top \mA$ operator.

Using the matrix notation in \eqref{eqn: short-hand notation}, we obtain 
\begin{equation*}
\int_\Omega \frac{\delta F^A[\rho + \varepsilon\eta ]}{\delta \rho}  \, \dd \eta (\theta)
= -\tr\left( (\mM + \varepsilon\mD)^{-2} \mD\right).
\end{equation*}
From the first-order approximation
\[
(\mM + \varepsilon \mD)^{-2} \approx \mM^{-2} - \varepsilon\mM^{-2}\mD\mM^{-1}- \varepsilon\mM^{-1}\mD\mM^{-2},
\]
we can write (eliminating the second order expansion on $\varepsilon$):
\begin{eqnarray}
 \int_\Omega \frac{\delta F^A[\rho + \varepsilon\eta ]}{\delta \rho}\, \dd \eta (\theta) 
\label{eqn: hessian A1} 
& \approx & - \tr\left(\left(\mM^{-2} - \varepsilon\mM^{-2}\mD\mM^{-1}- \varepsilon\mM^{-1}\mD\mM^{-2}\right) \, \mD\right)
\\
\nonumber
& = &  -\tr\left(\mM^{-2} \mD - 2\varepsilon \mM^{-1} \mD\mM^{-1}\mD\mM^{-1}\right),
\end{eqnarray}
which gives the first term of \eqref{eqn: hessian def}.
For the second term in \eqref{eqn: hessian def}, we have
\begin{equation}
\label{eqn: hessian A2}
\int_\Omega \frac{\delta F^A[\rho]}{\delta \rho}  \, \dd \eta (\theta) =-\tr\big((\mA^\top \mA [\rho])^{-2} \mA^\top \mA [\eta]\big)
=- \tr\left(\mM^{-2} \mD \right).
\end{equation}
By combining  \eqref{eqn: hessian A1} and \eqref{eqn: hessian A2} into  \eqref{eqn: def_hess_original}, we obtain the A-optimal Hessian formula:
\begin{equation}
\label{eqn: hessian A form}
\begin{array}{ll}
\hess~F^A[\rho] (\eta, \eta) & = \displaystyle 2\tr\left(\mM^{-1} \mD\mM^{-1}\mD\mM^{-1}\right).
\end{array}
\end{equation} 

We now prove that this Hessian operator is positive semidefinite by proving nonnegativity of \eqref{eqn: hessian A form}. 
By definition, the matrix
\[
\mM = \mA^\top \mA [\rho] = \int_\Omega \mA(\theta,:) \mA(\theta,:)^\top \dd \rho(\theta)  \in \Real^{d \times d}
\]
is positive definite, so $\mM^{-1}$ is also positive definite.  It follows that for any vector $\vz$, we have
\[
\vz^T\mM^{-1} \mD \mM^{-1} \mD \mM^{-1} \vz = (\mD \mM^{-1} \vz )^\top \mM^{-1} (\mD \mM^{-1} \vz) \ge 0,
\]
so $\mM^{-1} \mD \mM^{-1} \mD \mM^{-1}$ in \eqref{eqn: hessian A form} is positive semidefinite, as required.
The Hessian value \eqref{eqn: hessian A form} is non-negative, so $F^A[\cdot]$ is a convex functional. 

For $F^D[\cdot]$ defined in  \eqref{eqn: cont D-opt}, we have
\begin{equation}
\label{eqn: hessian D1}
\begin{array}{ll}
\displaystyle \int_\Omega \frac{\delta F^D[\rho + \varepsilon\eta ]}{\delta \rho}  \, \dd \eta (\theta) 
& \displaystyle = \tr\left( (\mM + \varepsilon\mD)^{-1}\mD\right)\\
& \displaystyle = \tr\left( \left(\mM^{-1} -\varepsilon\mM^{-1} \mD \mM^{-1} \right) \mD\right)+O(\varepsilon^2),
\end{array}
\end{equation}
and
\begin{equation}
\label{eqn: hessian D2}
\int_\Omega \frac{\delta F^D[\rho]}{\delta \rho} \, \dd \eta (\theta) 
=\tr\left(\mM^{-1} \mD \right).
\end{equation}
By combining \eqref{eqn: hessian D1} and \eqref{eqn: hessian D2}, we obtain
\[
\hess~F^D[\rho] (\eta, \eta) = -\tr\left(\mD\mM^{-1}\mD\mM^{-1}\right).
\]
Using an argument similar  to the one for $F^A$,  we can show that the D-optimal Hessian value 
\[
-\tr\left(\mD\mM^{-1}\mD\mM^{-1} \right) = -\tr\big( (\mM^{-1/2} \mD \mM^{-1/2})(\mM^{-1/2} \mD \mM^{-1/2})\big) \leq 0,
\]
for all $\eta$, where the nonnegativity of the trace follows from symmetry of $\mD$ and $\mM$ and thus of $\mM^{-1/2} \mD \mM^{-1/2}$.
Therefore, the D-optimal objective Hessian operator is negative semidefinite for measure $\rho$.
\end{proof}
We comment that the convexity of the objective function $F^A$ is shown in the linear fashion. This convexity property does not directly imply the geometric convexity in $W_2$.\footnote{For two measures $u, v \in \pr_2(\Omega),$ the $W_2$ space considers the displacement convexity notion (Section 3 of \cite{S15}), i.e.: 
$
F(T_t(\mu,\nu))\leq tF(\mu) +(1-t)F(\nu),
$
where the Wasserstein geodesic $T_t(\mu,\nu)$ replaces the classical convex interpolation: $T_t(\mu,\nu) = t\mu+(1-t)\nu$.} As a consequence, Proposition~\ref{prop: convexity} cannot be directly applied to show the convergence of the Wasserstein gradient flow \eqref{eqn: W2 flow}. Transferring the positive definite Hessian to the Wasserstein convergence requires more detailed undertaking that is beyond the current paper. We mention a few successful efforts~\cite{suzuki2023uniformintime,Chen_2025} carried out in different settings. (Similar comments apply regarding the concavity property of $F^D$.) Numerical observations presented in Section~\ref{sec: numerics} show that gradient flow can converge to different local optima when started from different initial points.

\subsection{Particle gradient flow simulation error}
\label{subsec: simulation error}
We turn now from examining the properties of the continuous OED formulation to the performance of Algorithm~\ref{alg: particle flow}. Multiple layers of numerical approximations are deployed in the algorithm, and rigorous convergence analysis is somewhat convoluted. 
Rather than present such an analysis fully, we identify the main source for the numerical error and provide a possible roadmap for a convergence analysis. We propose a conjecture about the convergence behavior, leaving detailed analysis to future research.

The error we aim to control is the difference between the global optimizer $\rho^*$, defined in~\eqref{eqn: cont A-opt} or \eqref{eqn: cont D-opt}, and the output of Algorithm~\ref{alg: particle flow} $\rho^T_{N,\dd t}$. 
(We have added subscript $\dd t$ to stress the dependence on time-discretization in the algorithm.) 
Since Wasserstein provides a metric that honors the triangle inequality, we deduce that
\begin{equation}
\label{eqn: convergence error}
W_2(\rho^*, \rho^T_{N, \dd t}) \leq W_2(\rho^*, \rho^T) +  W_2(\rho^T, \rho^T_{N}) + W_2(\rho^T_{N}, \rho^T_{N, \dd t}),
\end{equation}
where $\rho^T$ is the solution to the Wasserstein gradient flow \eqref{eqn: W2 flow} at time $T$ and $\rho^T_N$ denotes the particle approximation of $\rho^T$ using~\eqref{eqn: empirical measure}.
We expect all three terms are controllable under certain scenarios.
\begin{enumerate}
\item When the problem is geodesically convex, we expect $\rho^T \to \rho^*$ as $T \to \infty$. The nature of this convergence will be problem-dependent.
\item Replacing $\rho$ by $\rho_N$ amounts to replacing the continuous-in-space PDE by a finite number of samples. Intuitively, the more samples one pays to simulate the underlying flow, the more accurate the PDE solution becomes. Rigorously evaluating the difference is the main theme of mean-field analysis~\cite{FG15}. When the gradient flow is Lipschitz-smooth, it is expected that $\rho^T_N \xrightarrow{W_2} \rho^T$ as $N \to \infty$, with a potential rate of
\begin{equation}
\label{eqn: err est}
\E\, \left[W_2(\rho^T, \rho^T_{N})\right] \sim O \left(\frac{1}{N^\alpha} \right)\,,
\end{equation}
for $\alpha=\min\{2/\text{dim}(\Omega),1/2\}$. This convergence rate may be pessimistic in the sense that the rate is slow when $\text{dim}(\Omega)$ is high.

\item The discrepancy between $\rho^T_{N,\dd t}$ and $\rho^T_N$ is due to the discrete time stepping scheme. 
Following standard analysis of Euler's method \cite{A08}, the convergence is $\rho_{N,\dd t}^T \xrightarrow{W_2} \rho_N^T,$ as $\dd t \to 0$, with rate $W_2(\rho^T_{N} , \rho^T_{N, \dd t}) \sim O(\dd t)$ (see \cite[p.~69]{B16}).
\end{enumerate}
A similar analysis could be conducted for metrics other than the $W_2$ distance, such as TV norm or $\phi$-divergence (such as the KL divergence).    {We can also measure the weak convergence on a test function, possibly tightening the convergence in~\eqref{eqn: err est}. For any given $\psi\in C_c^\infty$, $\langle\rho^T-\rho^T_N,\psi\rangle\sim\frac{1}{\sqrt{N}}$. We refer to~\cite{PC19} for relations between different metrics.}

The error analysis allows us to be more explicit in calculating the computational complexity. In Algorithm~\ref{alg: particle flow}, we have noted that there are $T$ iterations, and in each iteration, $N$ particles are updated using the gradient information, a term that can be computed using either~\eqref{eqn: A velocity} for the  A-optimal measure or~\eqref{eqn: D velocity} for the D-optimal measure. 
Fortunately, the quantity $( \frac{1}{N} \sum_{i=1}^N \mA(\theta_i,:)^\top \mA(\theta_i,:))^{-1} \in \mathbb{R}^{d \times d}$ is shared by all simulated particles and thus needs to be computed just once per iteration. 
We summarize the cost per iteration here:
\begin{itemize}
    \item Computing $\sum_{i=1}^N \mA(\theta_i,:)^\top \mA(\theta_i,:)$: $O(d^2 N)$ cost;
    \item Computing the above matrix's inverse: $O(d^3)$;
    \item Assembling the particle velocity \eqref{eqn: A velocity} or \eqref{eqn: D velocity} (independently for $N$ particles): $O(Nd^2)$.
\end{itemize}
The total cost is thus $O((d^2N+d^3)T)$. 
The number of iterations $T$ is problem specific, depending on the convergence rate. 
The choice of $N$ depends on the mean-field rate, while $d$ is determined by the problem. 
If the required error tolerance is denoted by $\tau$ and  convergence in time is exponential, we have $T=|\ln\tau|$. 
If the metric of convergence is {\em weak} convergence, we have $N>\frac{1}{\tau^2}$ per Monte Carlo convergence rate. 
Under these conditions, the full cost is
\[
O\left(\left(\frac{d^2}{\tau^2}+d^3\right)|\ln\tau|\right)\,.
\]
For moderate $d$ and small error tolerance, the biggest computational bottleneck is to assemble a large number of gradients.

%% file: sections/5-setup.tex
\section{Optimal design model problem}
\label{sec: setup}
We demonstrate the optimal design setup for the case of of electrical impedance tomography (EIT) \cite{B02,U09}, a well-studied application from medical imaging. This section provides the preparation, including the inverse problem model and the OED setup, to support the subsequent OED tests on this model.

\subsection{EIT inverse problem and its linearization} 
The EIT experiment considers injection of a voltage into biological tissue and measurement of the electrical intensity on the surface (skin).
The problem is to infer the coefficient $\sigma$ in an inhomogenous elliptic equation from boundary measurements (Dirichlet and Neumann).
It is typically assumed that the biological tissue is close to a ground-truth medium, so linearization~\cite{C06} can be performed to recover the deviation from this ground-truth. 
The linearized problem solves the following equation for $\sigma$:
\begin{equation}
\label{eqn:fredholm_eit}
\int_{\mathcal{D}} r_\theta(y)\, \sigma(y) \dd y = \text{data}_\theta\,,
\end{equation}
where $r_\theta: \mathcal{D} \to \Real$ is a representative function. 
That is, when $r_\theta(y)$ is tested on $\sigma(y)$, it produces the data on the right hand side. 
The hope is that as one exhausts values of $\theta$, the testing function $r_\theta(y)$ spans the entire space $L_2(\mathcal{D})$, and the Fredholm first-kind integral problem \eqref{eqn:fredholm_eit} yields a unique reconstruction of $\sigma(y)$ in its dual space, which is also $L_2(\mathcal{D})$.

For this particular problem, the representative function $r_\theta(y)$ can be written explicitly as
\begin{equation}
r_\theta(y)= \nabla_{y} u(\theta_1, y) \cdot \nabla_{y} v(\theta_2, y),
\end{equation}
where $\theta=(\theta_1,\theta_2)$ represents the design point  and $u$ and $v$ solve the following forward and adjoint equations, respectively:
\begin{equation}
\label{eqn: forward-adjoint}
\begin{array}{l}
\displaystyle 
\text{forward~model~(voltage)}:~ \left\{ 
\begin{array}{l}
\nabla_y \cdot (\sigma \nabla_y u) = 0, ~~ y \in \mathcal{D}\\
u \vert_{\partial \mathcal{D}} = \mathbbm{1}_{\theta_1},
\end{array}
\right. \\\\
\displaystyle \text{adjoint~model~(intensity)}:~
\left\{ 
\begin{array}{l}
\nabla_y \cdot (\sigma \nabla_y v) = 0, ~~ y \in \mathcal{D} \\
v \vert_{\partial \mathcal{D}} = \mathbbm{1}_{\theta_2}
\end{array}
\right.
\end{array}
\end{equation}
Physically, this equation describes a voltage being applied at $\theta_1 \in \partial \mathcal{D}$ with electrical intensity collected at point $\theta_2 \in \partial\mathcal{D}$. The data on the right hand side of~\eqref{eqn:fredholm_eit} is the recording of this electrical intensity. 
The design space is therefore
\[
(\theta_1,\theta_2)\in\partial\mathcal{D}^2 =: \Omega.
\]

We set the computational domain $\mathcal{D}$ to be a unit disk in $\Real^2,$ so the boundary $\partial \mathcal{D}$ is a unit circle. We parameterize $\partial\mathcal{D}$ using $\theta_{1}, \theta_2 \in[0,2\pi]$, so that $\Omega=\partial\mathcal{D}^2=[0,2\pi]^2$ and is equipped with the standard 2D Lebesgue measure.
For a finite parametrization of the unknown medium $\sigma$, we discretize the integration domain $\mathcal{D}$, representing it in terms of a discrete mesh with a standard triangulation. Assuming continuous piecewise linear functions, the medium $\sigma$ is then represented by a finite-dimensional vector, with each entry corresponding to a value at a mesh vertex.
Using numerical quadrature, we reduce~\eqref{eqn:fredholm_eit} to  a linear system $\mA\vx = \vb$, where the vector $\vb$ takes on the value of $\text{data}_\theta$ from~\eqref{eqn:fredholm_eit}. 
The continuous experiment/data  matrix $\mA$ has $d$ columns, with rows  indexed by a particle pair $(\theta_1, \theta_2) \in \partial \mathcal{D}^2$, that is,
\begin{equation}
\label{eqn: EIT measurement}
\underbrace{[\dots \nabla_y u(\theta_1, y_j) \cdot\nabla_y v(\theta_2, y_j)\dots]}_{\mA\left(\left(\theta_1,\theta_2\right), :\right)}
\underbrace{\left[ 
\begin{array}{c}
     \vdots  \\    
    \sigma( y_j) \, \Delta y_j \\    
     \vdots     
\end{array}
\right]}_{\vx}.
\end{equation}
Note that $u$ and $v$ are solutions to equations parameterized by $\sigma$, so different values of the ground truth $\sigma$ would lead to different matrices $\mA$.

\begin{remark}
To prepare the continuously indexed matrix $\mA$, we discretize the boundary domain $\partial \mathcal{D}$ into a finite collection of nodes and simulate the forward and adjoint models \eqref{eqn: forward-adjoint} with these discretized boundary conditions. 
For particles $(\theta_1, \theta_2)$ between the nodes, we use linear interpolation to approximate the solutions $u(\theta_1, :)$ and $v(\theta_2, :).$ For example, when $\theta_1 \in [\theta_{1}^L, \theta_{1}^R]$, where  $\theta_{1}^L$ and  $\theta_{1}^R$ represent the nearest nodes from the discretization that are left and right of $\theta_1$, respectively, we approximate the forward model solution by
\[
u(\theta_1,:) = \frac{\theta_{1}^R - \theta_1}{\theta_{1}^R - \theta_{1}^L} u\left(\theta_{1}^L, :\right) + \frac{\theta_{1} - \theta_{1}^L}{\theta_{1}^R - \theta_{1}^L} u(\theta_{1}^R,:)\,. 
\]
\end{remark}

\subsection{EIT optimal design in the linearized setting}
The design problem associated with the EIT example is to find the optimal sensor placement that  coordinates voltage injection on $\theta_1$ with electricity measurement $\theta_2$ on the surface $\partial \mathcal{D}$. 
Mathematically, we solve for a bivariate probability distribution $\rho(\theta_1, \theta_2)$ with $(\theta_1,\theta_2)\in\partial\mathcal{D}^2$ that optimizes the OED criteria \eqref{eqn: cont A-opt} and \eqref{eqn: cont D-opt}.

For our tests, we consider two cases, where the ground-truth media $\sigma: \mathcal{D} \to \Real$ is homogeneous in the first case and inhomogeneous in the second case.
\begin{enumerate}
\label{list}
    \item Homogeneous media: 
    \begin{equation}   
    \label{eqn: homo}
    \sigma(y) \equiv c>0, \quad \forall y \in \mathcal{D}.
    \end{equation}
    \item Inhomogeneous media:  
    \begin{equation}
    \label{eqn: inhomo} 
   \displaystyle \sigma(y) = 
    7+ 50\,\text{exp}\left\{ -\frac{(y(1) -\frac{1}{3})^2 +(y(2) -\frac{1}{3})^2}{2\,(\frac{1}{10})^2} \right\},  \quad \forall y \in \mathcal{D}. 
 \end{equation}
    \end{enumerate}
These media layouts are depicted in the first row of Figure~\ref{fig: landscape_A}. 
The associated data matrices $\mA$ \eqref{eqn: linear model} are denoted by $\mA_{\hom}$ and $\mA_{\ihom}$, respectively.

We apply the finite-element method for EIT discretization and simulation to obtain the continuous experiment/data matrix of \eqref{eqn: EIT measurement}. 
The number of design nodes $\theta$ on $\partial \mathcal{D}$ is $n=200$, equally spaced on the unit circle with angular gap of $2\pi/200$. The number of the interior nodes in the domain $\mathcal{D}$ is $d=20$, which is the number of unknown parameters required to specify the media $\sigma$.
The finite matrix $\mA$ \eqref{eqn: linear model} therefore has dimensions $200^2 \times 20$. We compute the derivatives $\partial_{\theta_1} \mA, \partial_{\theta_2} \mA$ using forward finite differences. Across all the numerical tests below in Sections~\ref{sec: numerics} and~\ref{sec: new numerics}, for the realization of a probability distribution $\rho$, we sample $N = 10000$ particle pairs $(\theta_1, \theta_2)$ from design space $\partial \mathcal{D}^2$.
All our figures show averaged results from 10 independent simulations.  

To prepare to run OED, we first analyze the landscapes of the objective functions \eqref{eqn: cont A-opt} and \eqref{eqn: cont D-opt},  plotting $F^A$ and $F^D$ as functions of $\rho$. 
To do this, we first need to confine $\rho$ to a parameterized subset. Define a sequence of $\rho_L$ (parameterized by $L$) as
\begin{equation}
\label{eqn:rho_parameter_plot}
\rho_L := U_{\{|\theta_1-\theta_2|\leq U[0, L]\}}
\end{equation}
(a uniform distribution over the set $|\theta_1-\theta_2|\leq L$).

In Figure~\ref{fig: landscape_D}, we plot $F^A[\rho_L]$ and $F^D[\rho_l]$ as a function of $L$ for different media. The result is quite revealing even in this simple setup. We first compare (A) and (B) in Figure~\ref{fig: landscape_D}, the landscapes for different media. For uniform media, $F^A$ is entirely above the line predicted by uniform distribution, and reaches its minimum at $L=\pi$, where $\rho_L$ coincides with the uniform distribution. This aligns with our intuition: For homogeneous media, sensors and sources should be placed in an homogeneous fashion. When media is inhomogeneous, however, $F^A$ is entirely below the line predicted by uniform sampling, meaning uniform sampling is the worst distribution in the class considered. Furthermore, $F^A$ picks its minimum at $L={\pi}/{4}$, suggesting the optimal $\rho$ should have sensor and source closely aligned in a window not more than ${\pi}/{4}$ apart.

Figure~\ref{fig: landscape_D} (C) depicts $F^D[\rho_L]$ as a function of $L$. For $F^D$, we aim to find the maximum value, and in this class of probabilities, the optimal solution is $L=0$, meaning the voltage injection and intensity measurements are placed exactly at the same location. This is in stark contrast with panels (A), (B), highlight the fact that different objective criteria can produce significantly different solutions.

\begin{figure}
\centering
\subfloat[Homogeneous]{\includegraphics[scale=0.45]{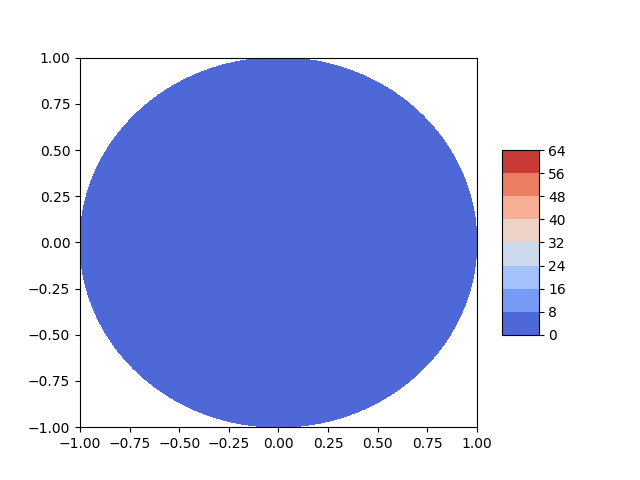}}\hspace*{-18mm}
\subfloat[Inhomogeneous]{\includegraphics[scale=0.45]{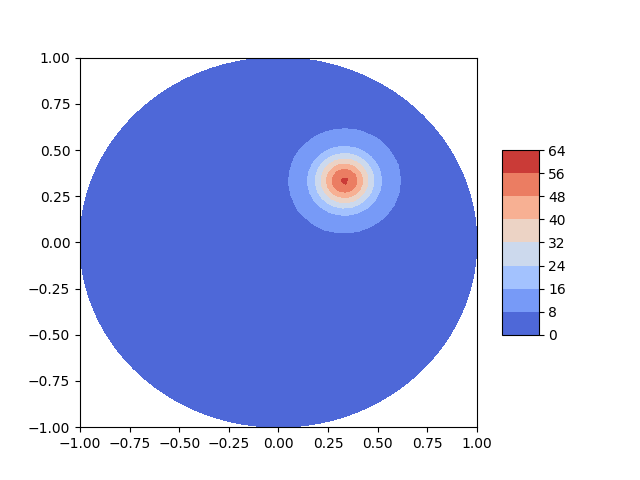}}\hspace*{-10mm}
\caption{Media configurations for homogeneous media \eqref{eqn: homo} and inhomogeneous media \eqref{eqn: inhomo}, respectively.}
\label{fig: landscape_A}
\end{figure}

\begin{figure}[!htb]
\centering
\subfloat[]{\includegraphics[scale=0.3]{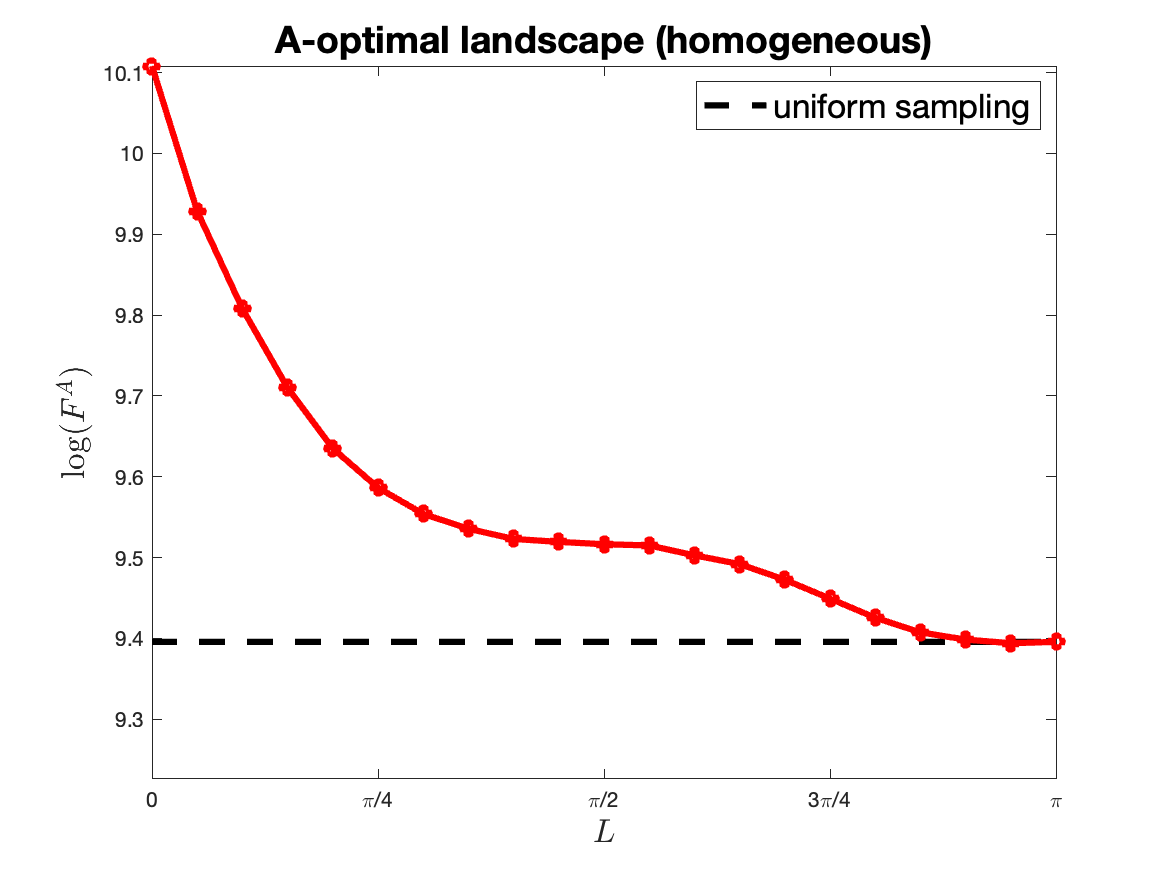}}
\hspace*{-5mm}\subfloat[]{\includegraphics[scale=0.3]{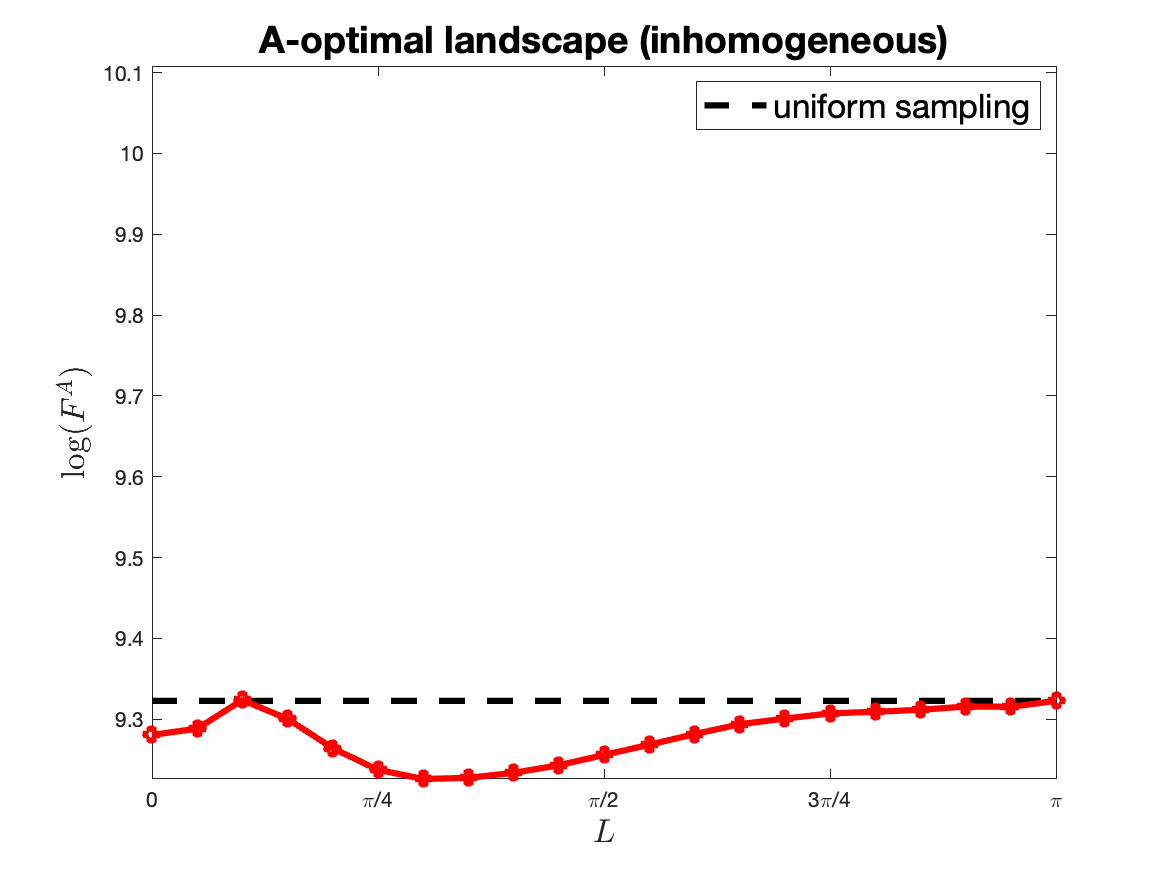}}
\hspace*{-5mm}\subfloat[]{\includegraphics[scale=0.3]{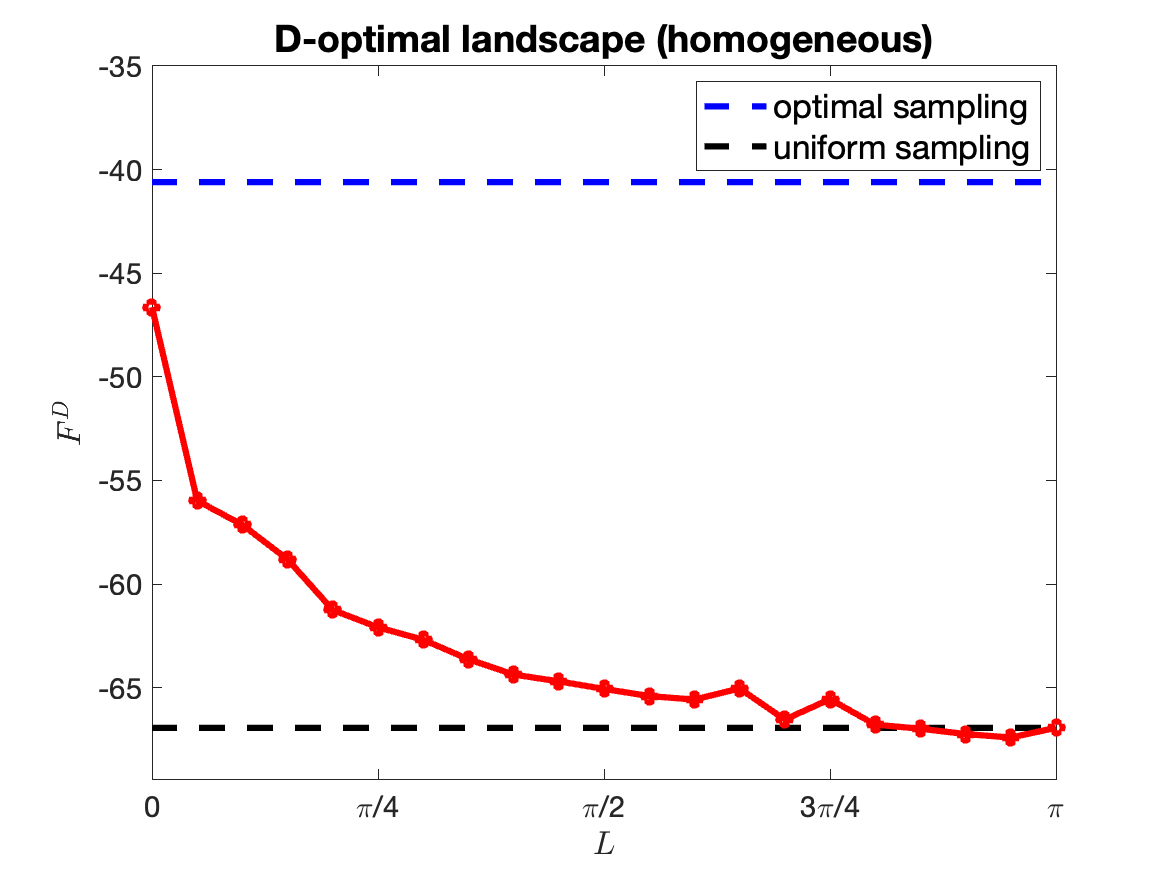}}
\caption{Panels (A), (B) show the landscape of the A-optimal objective \eqref{eqn: cont A-opt}, respectively tested on homogeneous \eqref{eqn: homo} and inhomogeneous \eqref{eqn: inhomo} media. Panel (C) is the landscape of the D-optimal objective \eqref{eqn: cont D-opt} in homogeneous case. The landscapes are captured by $\rho_L$ defined in \eqref{eqn:rho_parameter_plot}. For reference, the black dashed lines are obtained from a uniform sampling distribution over the entire $\Omega = \partial \mathcal{D}^2$. For (C), we also plot the solution to the optimal sampling strategy achieved by the classical Fedorov method \cite{F13,JD75}.}
\label{fig: landscape_D}
\end{figure}

%% file: sections/6-numerics.tex
\section{Gradient flow for linearized EIT design}
\label{sec: numerics}

We now describe the numerical performance of the particle gradient flow Algorithm~\ref{alg: particle flow} on EIT inverse problem, with subsections devoted to A-optimal and D-optimal design, respectively.

\subsection{A-Optimal design}
\label{subsec: A}

For the A-optimal objective criterion \eqref{eqn: cont A-opt}, our investigation focuses on two crucial factors that influence the resulting OED converging patterns: 1. the homogeneity of the underlying EIT medium, encoded through data matrices $\bf A_{\hom}$ \eqref{eqn: homo} and $\bf A_{\ihom}$ \eqref{eqn: inhomo}; 2. the initialization strategy of the particle gradient flow.
Particularly, we define the initialization $\rho^0$ to be a uniform distribution supported on one of three regions:
\begin{enumerate}
\renewcommand{\labelenumi}{\textbf{\theenumi}}
\renewcommand{\theenumi}{Init.\arabic{enumi}}
\makeatletter
\makeatother
\item 
\label{init: entire space}
entire design space; 
\item 
\label{init: L-shape}
restricted L-shape; 
\item 
\label{init: diagonal}
diagonal stripe. 
\end{enumerate}
(See Figure~\ref{fig: sampling region} for visualizations of the initialization.)

The design of~\ref{init: L-shape} assigns heavier weights to samples for which either $\theta_1$ or $\theta_2$ lie in the sector $[0,\frac{\pi}{2}]$. 
According to~\eqref{eqn: inhomo}, this region could carry more information, suggesting that it is a good place for sources and detectors. 

\begin{figure}[!htb]
\centering
\includegraphics[scale=0.37]{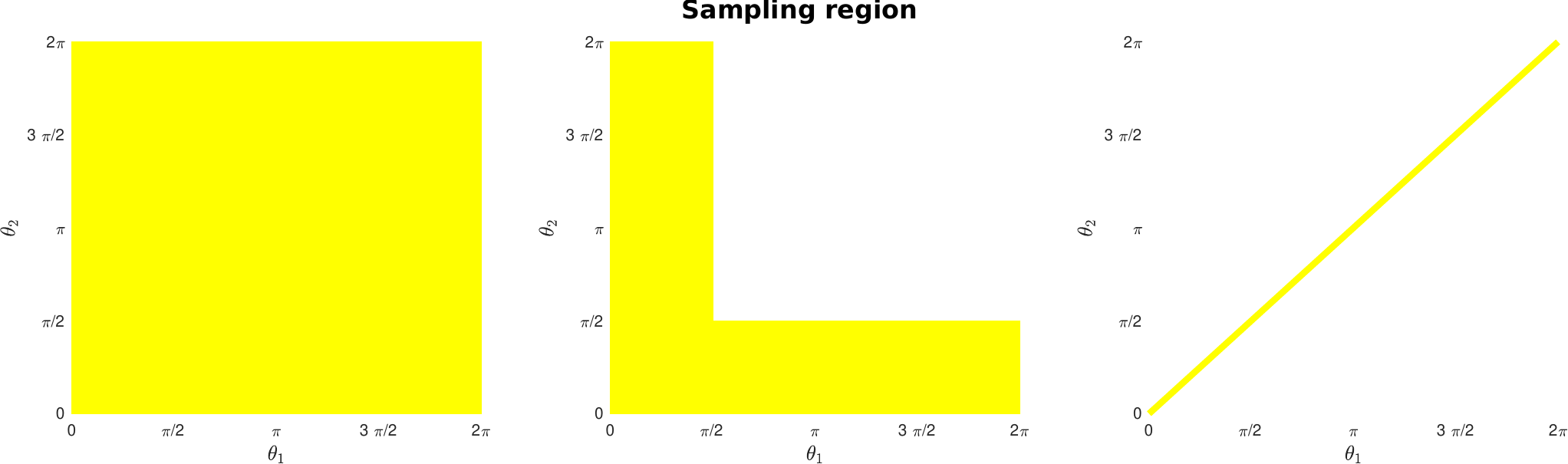}
\caption{Initialization strategies for particle sampling, illustrating regions \ref{init: entire space}, \ref{init: L-shape}, \ref{init: diagonal}, respectively.}
\label{fig: sampling region}
\end{figure}

We first examine the gradient information. Defining $\rho$ as in \ref{init: entire space}, we compute:
\[
\nabla_\theta\frac{\delta F^A[\rho]}{\delta\rho} \in \Real^2 \quad \eqref{eqn: A velocity}
\]
as a function of $\theta = (\theta_1, \theta_2)$ over the design space. This function is plotted in Figure~\ref{fig: A velocity field}, with the left panel showing results for homogeneous media $\mA_\hom$ and the right panel showing those for inhomogeneous media $\mA_{\ihom}$.
For $\mA_\hom$ (Figure~\ref{fig: A velocity field}(A)), the gradient magnitude is rather balanced over the entire design space, with relatively higher magnitude near the diagonal, where $\theta_1 \approx \theta_2$. 
In contrast,  Figure~\ref{fig: A velocity field}(B) shows that $\mA_{\ihom}$ has much stronger disparity in in the gradient, with the highest magnitude seen in the region of $(\theta_1,\theta_2)\in[0,\frac{\pi}{2}]^2$. 

\begin{figure}
\centering
\subfloat[Homogeneous]{\includegraphics[scale=0.35]{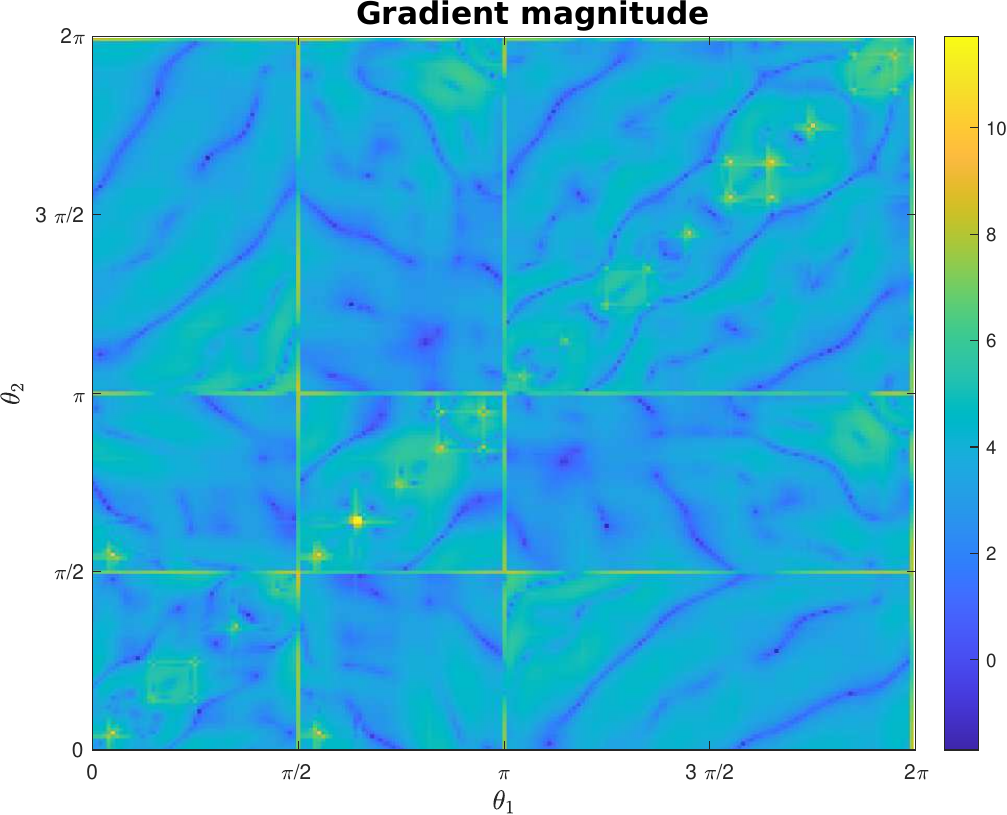}}
\subfloat[Inhomogeneous]{\includegraphics[scale=0.35]{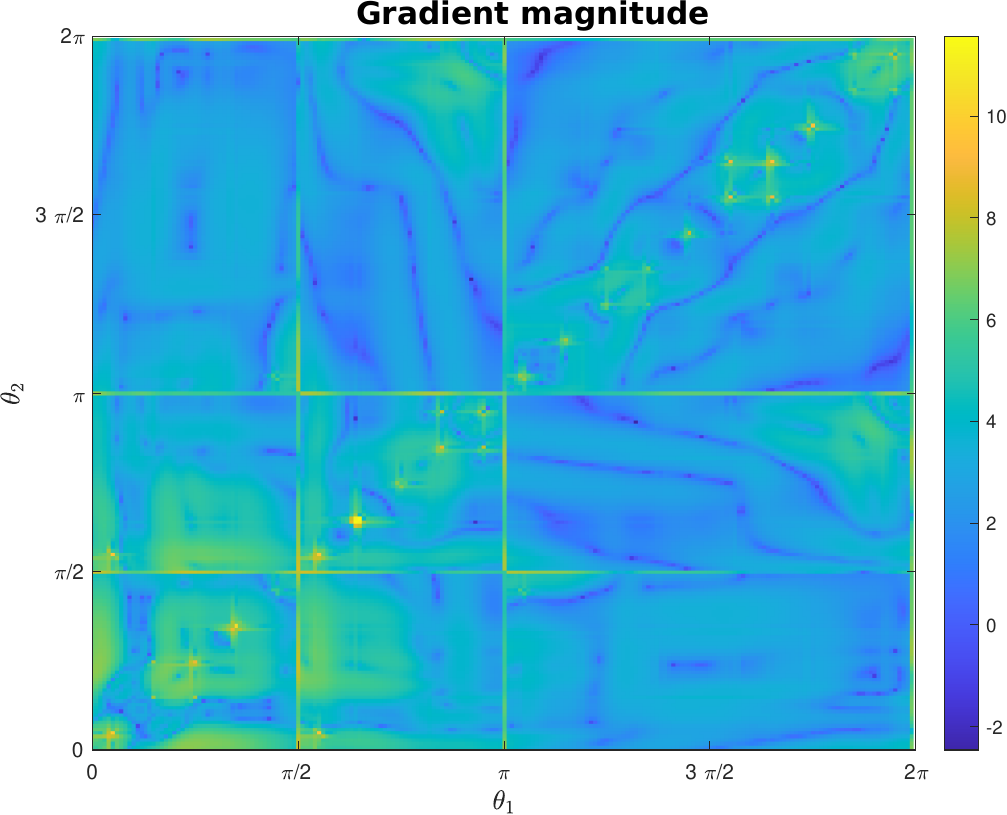}}
  \caption{The gradients are computed using the formula~\eqref{eqn: A velocity}. (A) and (B) respectively show the gradient magnitude for the homogeneous case $\mA_\hom$ and inhomogeneous case $\mA_\ihom$.} 
    \label{fig: A velocity field}
    \end{figure} 

We now discuss the evolution of probability measure $\rho$ during Algorithm~\ref{alg: particle flow}  for different scenarios of media types and initializations. 
For gradient flow, we set the time discretization to $\dd t = 10^{-3}$ and the total iteration number to $T = 5000.$
A periodic boundary condition is deployed in Algorithm~\ref{alg: particle flow}. 
To be specific, in line 3 of Algorithm~\ref{alg: particle flow},
the updated particle location follows $\theta$ is set to be $(\theta \mod 2\pi)$.

\begin{table}[!htb]
\begin{center}
\begin{adjustbox}{width=.65\columnwidth,center}
\begin{tabular}{c|ccc} 
  \hline
&  \ref{init: entire space}& \ref{init: L-shape}&  \ref{init: diagonal}\\ 
\hline
  homogeneous  & maintain & prone to spread & maintain \\ 
 inhomogeneous  & L-shape highlighting  & maintain & $/$\\
  \bottomrule
\end{tabular}
\end{adjustbox}
\end{center}
\caption{Gradient flow convergence summary for A-optimal design. Rows represent media; columns represent the three initialization schemes.}\label{tab:convergence}
\end{table}

Convergence results are summarized in Table~\ref{tab:convergence}. Further details for the case of homogeneous media are shown in Figure~\ref{fig: A homo}. Each row of panels shows evolution from one of the initialization schemes, including the initialization on the left and an advanced time point in the third panel.
For all initial sampling regimes, the objective values (shown in the rightmost column) decay until they saturates at a plateau, consistent with known properties of the gradient flow. When the initial measure is uniform in the entire domain $[0,2\pi]^2$, the algorithm returns a distribution that again spreads the domain. When the distribution is initially confined to the L-shaped area, it tends to move beyond the local region and spread further.
We note from panel (c) that the distribution concentrated along the diagonal seems to be a local minimum: If all samples are initially prepared in the diagonal stripe (\ref{init: diagonal}), the gradient flow only moves them along the diagonal, producing a final probability distribution supported only on the diagonal $\theta_1\approx \theta_2$. 
Note that the final value of $F^A$ is considerably larger here than for the other initialization schemes, suggesting this is a local optimum. 

\begin{figure}
\centering
\subfloat[\ref{init: entire space}.]{\includegraphics[scale=0.27]{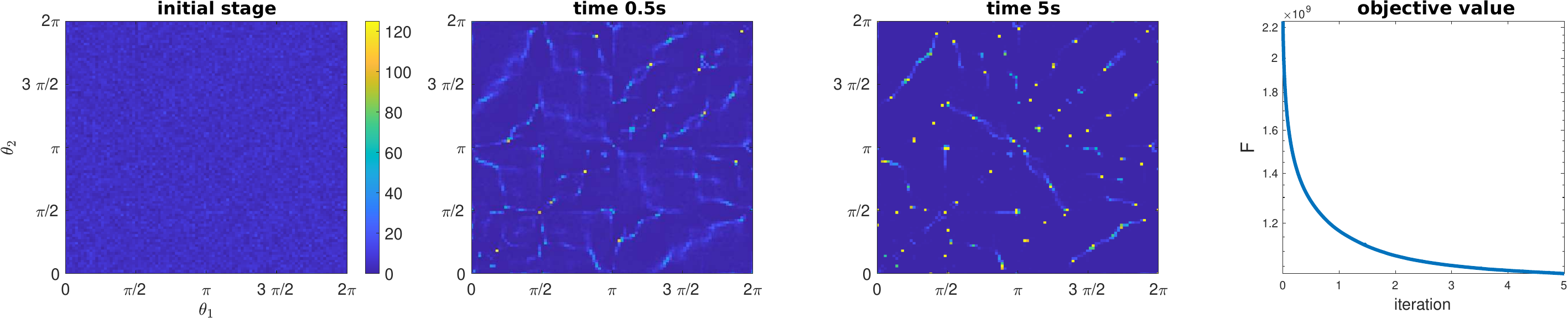}}\\
\centering
\subfloat[\ref{init: L-shape}. ]{\includegraphics[scale=0.27]{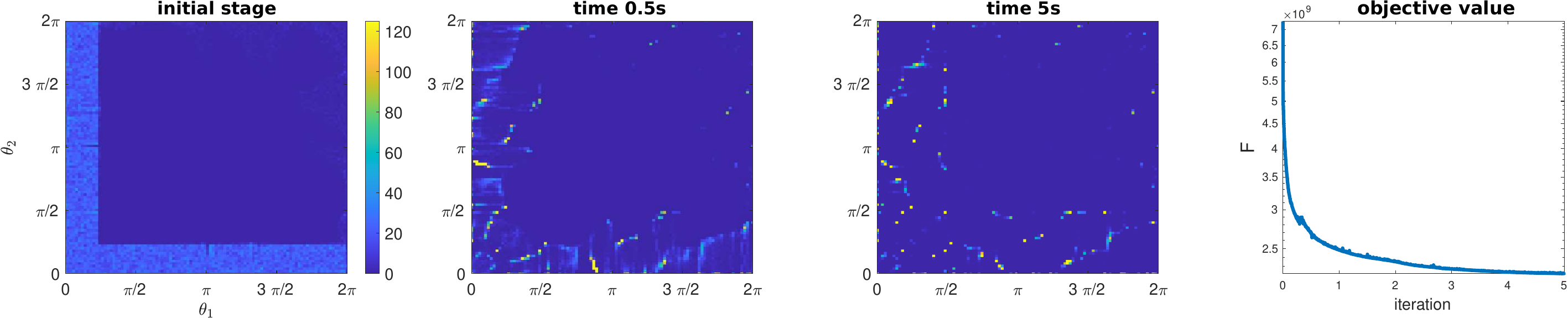}}\\
\centering
\subfloat[\ref{init: diagonal}. ]{\includegraphics[scale=0.27]{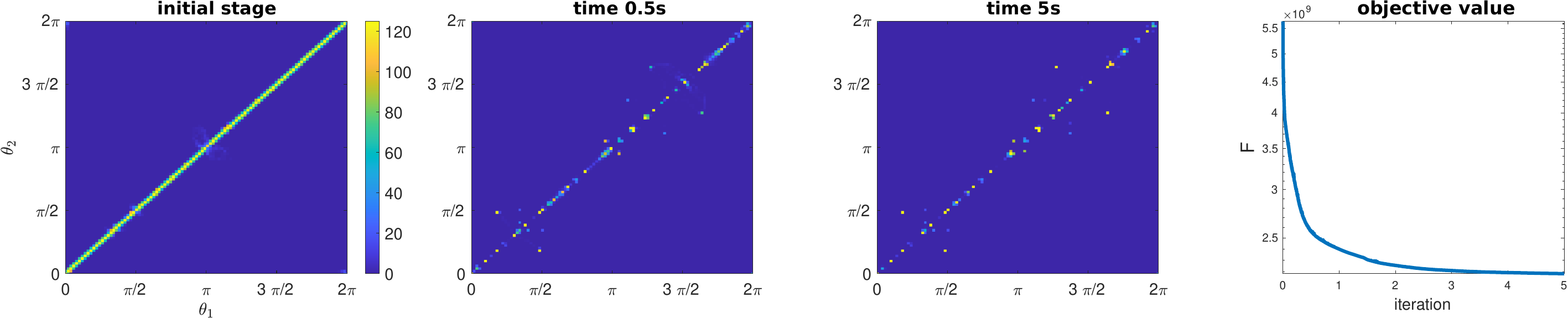}}\\
\caption{Homogeneous regime: Evolution of particle gradient flow Algorithm~\ref{alg: particle flow} under A-optimal criterion. }
\label{fig: A homo}
\end{figure}

Figure~\ref{fig: A inhomo} illustrates a similar study for the inhomogeneous media $\mA_{\ihom}$. Similar to the homogeneous example, the objective function value $F^A$ decreases steadily, but the final configurations are quite different from the homogeneous case. 
In the first row (\ref{init: entire space}), particles initially sampled over the entire space tend to highlight the restricted L-shape  as well as other scattered parts.
In the second row (\ref{init: L-shape}), where the initial samples are in the L-shaped area already, they tend to stay in that region. These results suggest that either the source $\theta_1$ or the detector $\theta_2$ should be placed within the angle $[0,\pi/2],$ as this region delivers more information than the rest of the domain. 
Similar to the homogeneous media case, the sampling concentrated along the strip of $\theta_1=\theta_2$ appears to represent a local minimum, with an initial distribution with this property leading to subsequent iterates sharing the same property. (We omit the plots for this case.)

\begin{figure}[!htb]
\centering
\subfloat[\ref{init: entire space}]{\includegraphics[scale=0.27]{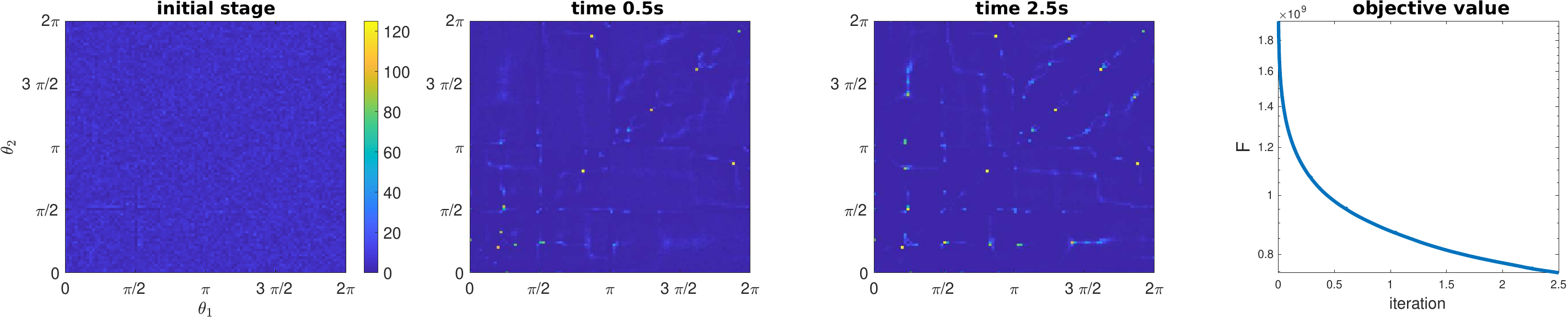}}\\
\centering
\subfloat[\ref{init: L-shape}]{\includegraphics[scale=0.27]{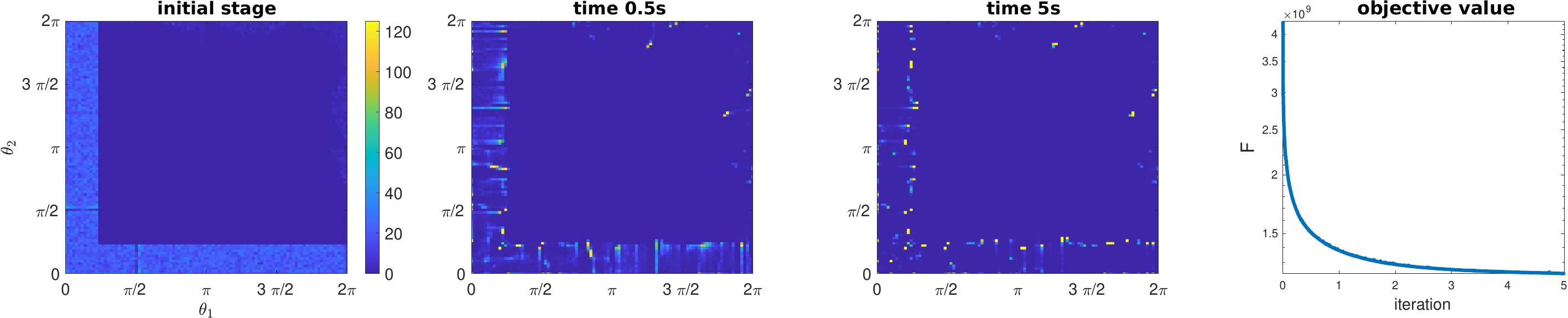}}
\caption{Inhomogeneous regime: evolution of particle gradient flow Algorithm~\ref{alg: particle flow} under A-optimal criterion.}
\label{fig: A inhomo}
\end{figure}

In summary of A-optimal design for two media layouts, we observe a recurring theme: Uniform sampling across the entire space and diagonally alignment generally represent (local) optimal design measures. In the presence of inhomogeneity, another favorable region emerges where either the source or detector is positioned near the medium bump ($\theta_i\in[0,{\pi}/{2}]$).

\subsection{D-Optimal design}

In this section, we investigate the design performances of Algorithm~\ref{alg: particle flow} using the D-optimal design criterion \eqref{eqn: cont D-opt}. The numerical tests conducted here are less extensive than those in the previous A-optimal Section~\ref{subsec: A}. We primarily focus on assessing the influence of the initialization strategy on the D-optimal design results. In addition, we offer insights into whether D-optimal designs exhibit behavior related to that observed under the A-optimal criterion.
We consider two types of initialization strategies for the distributions of the samples particles: 
\begin{enumerate} 
\renewcommand{\labelenumi}{\textbf{\theenumi}}
\renewcommand{\theenumi}{Init.\alph{enumi}}
\makeatletter
\makeatother
    \item 
    \label{init: uniform}
    uniform distribution on the entire design space;
    \item
    \label{init: quasi-optimal}
    (approximately) optimal distribution provided by Fedorov method \cite{F13}.
\end{enumerate}

The gradient information of D-optimal design is shown in Figure~\ref{fig: D velocity field}. 
The gradient magnitude heatmap shows particles are the most active in the diagonal area. 

\begin{figure}[!htb]
\centering
\subfloat{\includegraphics[scale=0.35]{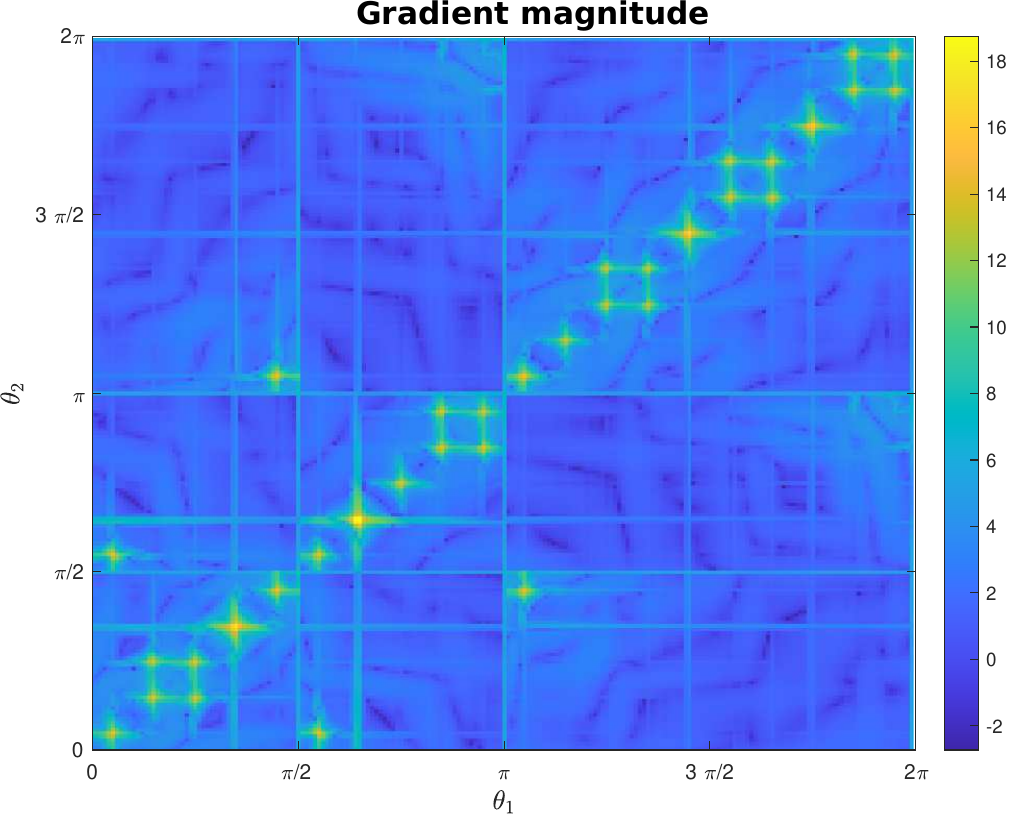}}
\caption{Gradient magnitude under D-optimal design criterion \eqref{eqn: cont D-opt}.}
\label{fig: D velocity field}
\end{figure}

Results for Algorithm~\ref{alg: particle flow}, using flow simulation time step $\dd t = 10^{-6}$, are shown in Figure~\ref{fig: D evolution}. (A) shows that an initially uniform distribution of $\rho$ largely remains uniform during Algorithm~\ref{alg: particle flow}, developing only thin concentration by the end at $T = 10^5$ steps. Meanwhile, the D-optimal value in the right-most panel demonstrates a significant improvement in the conditioning of the data matrix, highlighting the advantage of the OED particle gradient flow at this sensitive local optimum.
In Figure~\ref{fig: D evolution} (B), for the Fedorov-method initialization, we see again that the initial and final distributions are fairly similar in character, while again the D objective increases steadily showing the functionality of design.
This contrasting phenomenon further confirms the sensitivity of the EIT inverse problem, where conditioning can remain fragile even after a visually satisfactory design is obtained via the quasi-optimal Fedorov method. At the same time, it corroborates the value of our proposed gradient flow approach, which rigorously drives the design toward optimality, improving conditioning and supporting robust data acquisition.
Interestingly, the A-optimal objective function $F^A$ decreases steadily during execution for the first initialization, but not the second.
On both optimality criteria, the diagonally concentrated configuration (B) achieves better quantitative results than uniform distribution (A). Hence, for EIT medium with homogeneity, we recommend placing the source and detector in alignment with each other.

\begin{figure}[!htb]
\centering
\subfloat[\ref{init: uniform}]{\includegraphics[scale=0.27]{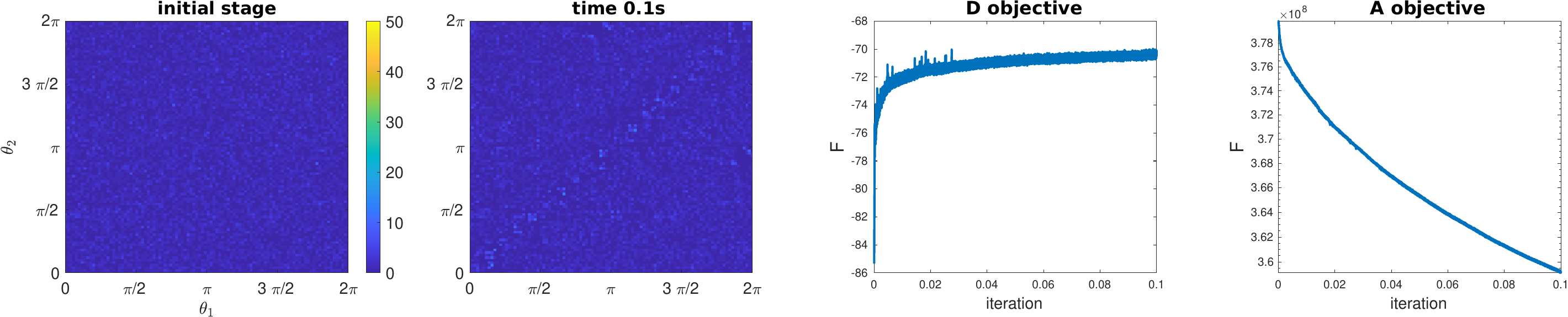}}\\
\centering
\subfloat[\ref{init: quasi-optimal}]{\includegraphics[scale=0.27]{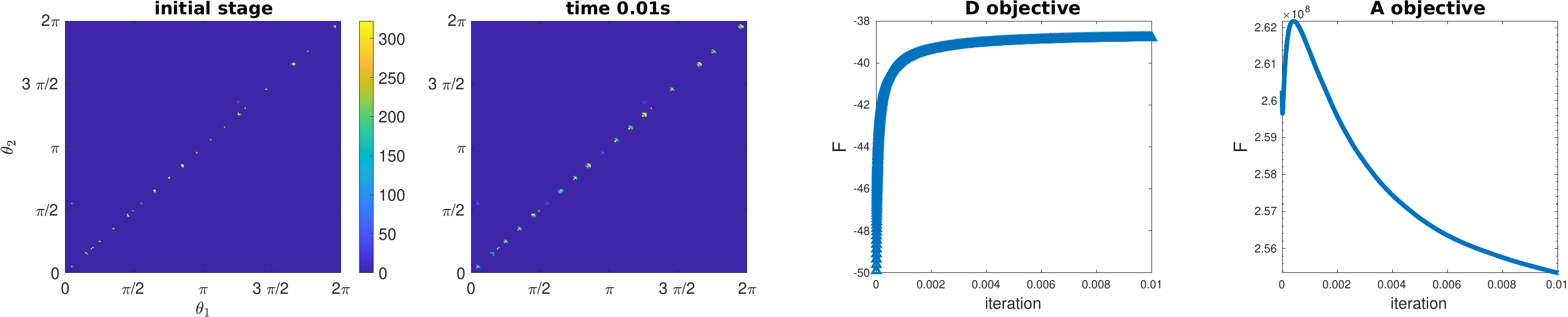}}
\caption{Evolution of particle gradient flow \ref{alg: particle flow} under D-optimal criterion, tested on the homogeneous media data $\mA_\hom$.}
\label{fig: D evolution}
\end{figure}

To conclude, we investigated A- and D-optimal continuous designs for the EIT model under varying data homogeneities using the particle gradient flow Algorithm~\ref{alg: particle flow}. A key finding is that the resulting OED is highly sensitive to initialization. Effective initializations can be obtained from prior physical knowledge of the model (see Figure~\ref{fig: A inhomo} (B)), or partial runs of classical quasi-optimal design methods, as in Figure~\ref{fig: D evolution} (B). Alternatively, uniform initialization over the design space can be used, though it may provide only a coarse indication of measure concentration, as illustrated in Figure~\ref{fig: A inhomo} (A).

\section{Linearized Darcy flow}
\label{sec: new numerics}
This section presents a second test example based on the 1D Darcy flow model \cite{HV19,GHLS20}. Compared to the previous EIT model, the Darcy flow example exhibits greater numerical stability and distinct physical characteristics. Our main goal is to evaluate the gradient flow OED Algorithm~\ref{alg: particle flow} on this model and to provide a complementary numerical study, along with interpretations that highlight the underlying physical insights.
 
The Darcy flow PDE is a 1D elliptic equation:
\begin{equation}
\label{eqn:Darcy}
\left\{
\begin{array}{l}
-(\sigma u')' = S(y),  ~~ y \in [0,1]    \\
u \vert_{y = 0, 1} = 0.
\end{array}\right.
\end{equation}
For every source $S(y)$, the solution $u$ can be measured in the entire domain $[0,1]$. The goal is to design $S(y)$ and measurement locations so as to reconstruct $\sigma$. Analogous to the EIT model, we employ a finite equispaced discretization represent $\sigma$ as an unknown vector in dimension $d.$ 

Similar to~\eqref{eqn:fredholm_eit}, the linearized problem is transformed into a Fredholm first type integral:
\[
\int_0^1 r_\theta(y) \sigma(y)\, \dd y = \text{data}_{\theta}\,,
\]
with $r_\theta(y)=u_{\theta_1}' (y) \, v_{\theta_2}'(y)$ where $u_{\theta_1}$ and $v_{\theta_2}$ solve the following forward and adjoint equations respectively:
\begin{equation}
\label{eqn: new model} 
\left\{
\begin{array}{l}
-(\sigma u')' = \delta_{\theta_1},  ~~ y \in [0,1]    \\
u \vert_{y = 0, 1} = 0.
\end{array}\right.
\quad 
\left\{
\begin{array}{l}
-(\sigma v')' = \delta_{\theta_2},~~ y \in [0,1]\\
v \vert_{y = 0, 1} = 0.
\end{array}\right.
\end{equation}
The pair $(\theta_1,\theta_2)$ belongs to the design space $\Omega = [0,1 ]^2$.

Computationally, we set the ground-truth media $\sigma: [0,1] \to \mathbb{R}$ to be a Gaussian function:
\[
\sigma(y) = 1+1000\, \text{exp} (-1000 (y-0.25)^2);
\]
see plot in Figure~\ref{fig: new model}(A). 
This contrast media produces PDE solutions of disparate profiles depending on the source location. In Figure~\ref{fig: new model}(B), we showcase the profile of $u$ \eqref{eqn:Darcy} with $\theta_1=0.25$ (coinciding  with the bump in the media) and $\theta_1=0.5$ (away from the bump). 
\begin{figure}[!htb]
\centering
\subfloat[]{\includegraphics[scale=0.35]{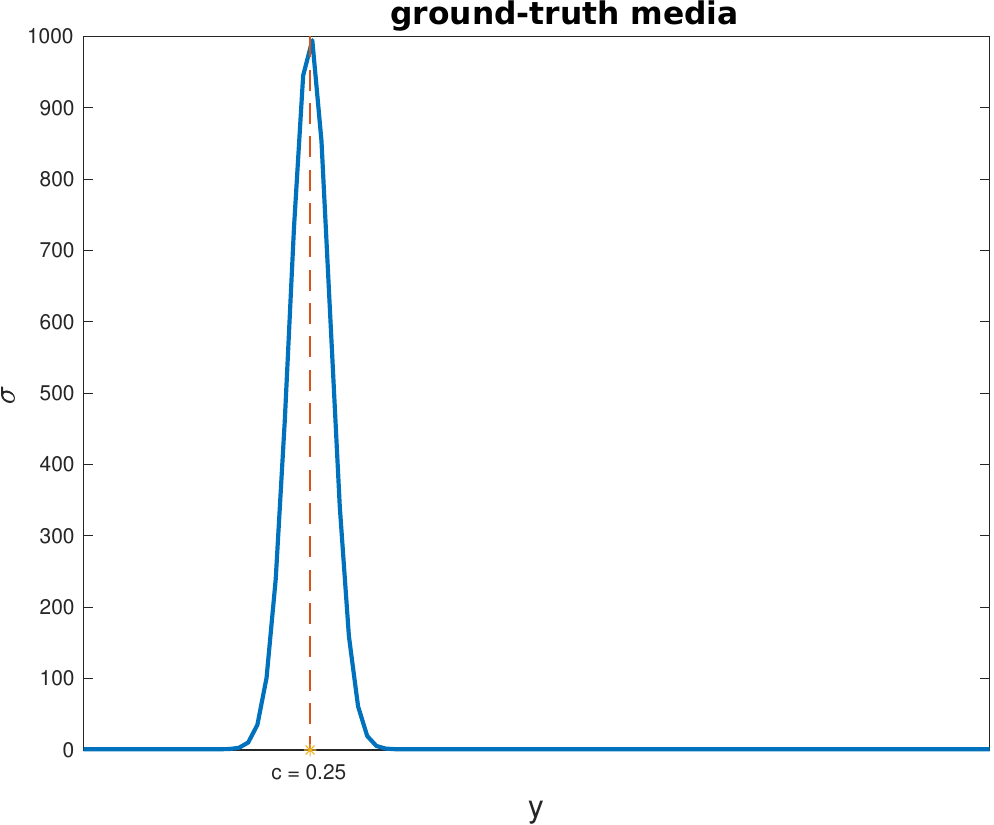}}~~
\subfloat[]{\includegraphics[scale=0.35]{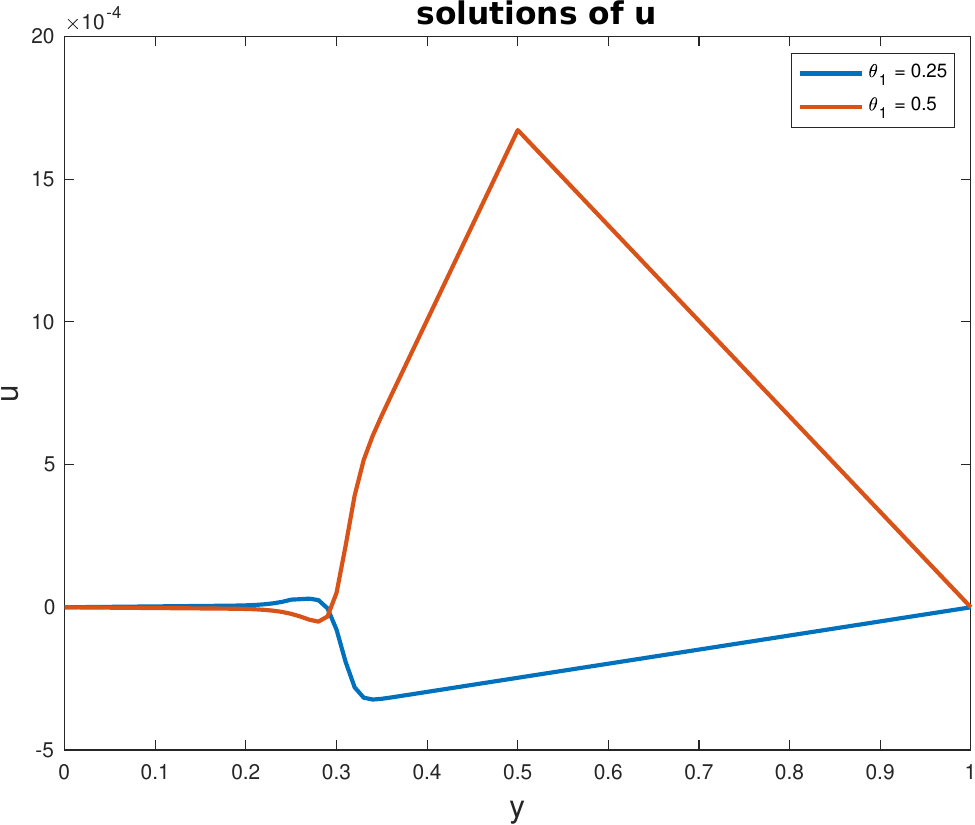}}
\caption{(A): The ground-truth media function $\sigma$ with the bump center at $c = 0.25$. (B): The forward model \eqref{eqn: new model} solutions of $u$ with different sources.}
\label{fig: new model}
\end{figure}

Numerically, we simulate the PDE model with $100$ equally spaced points on $[0,1]$, with $\dd y= 0.01$. 
We assume that $\sigma$ is piecewise constant and parametrize it using a $d$-dimensional vector. In this way, we construct the data matrix $\mA$ of size $100^2\times 20$ $(d=20)$.

Given this matrix $\mA$, the OED aims to find a probability distribution $\rho(\theta_1, \theta_2)$ with $(\theta_1,\theta_2)\in[0,1]^2$ that optimizes the A-optimal and D-optimal criteria in \eqref{eqn: cont A-opt} and \eqref{eqn: cont D-opt}, respectively, using Algorithm~\ref{alg: particle flow}.

In the A-optimal case, we first inspect the gradient magnitude \eqref{eqn: A velocity} in Figure~\ref{fig: A gradient new}, testing on the probability measure $\rho$ to be the uniform distribution over the entire design space $[0,1]^2.$ 
The particles are most active in the neighborhood of $(\theta_1,\theta_2)=(0.25,0.25)$, aligning with the media bump at $y=0.25$. 

\begin{figure}[!htb]
\centering
\includegraphics[scale=0.55]{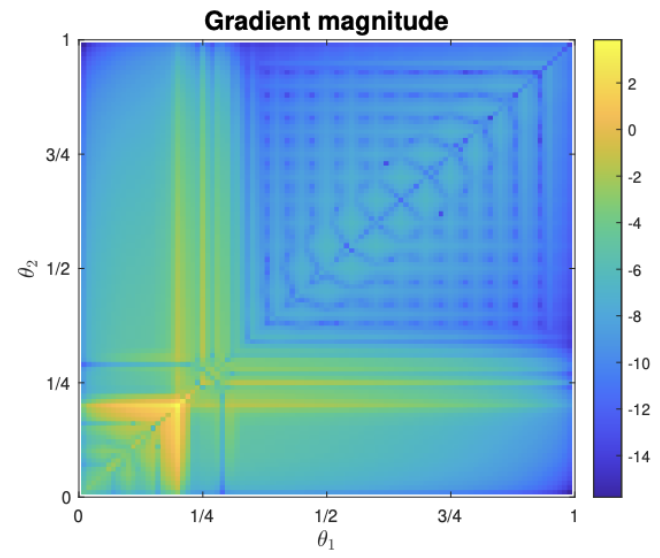}
\caption{Gradient magnitude under the A-optimal objective.}
\label{fig: A gradient new}
\end{figure}

Next, we demonstrate the progression of Algorithm~\ref{alg: particle flow}. Three snapshots of $\rho$ along the evolution are exhibited in Figure~\ref{fig: A evolution new}, with the starting point in the left panel and the A-optimal distribution at the right. Starting from a uniform distribution, the gradient flow drives the samples of $\rho$ away from the media bump at $0.25.$

\begin{figure}[!htb]
\centering
\includegraphics[scale=0.42]{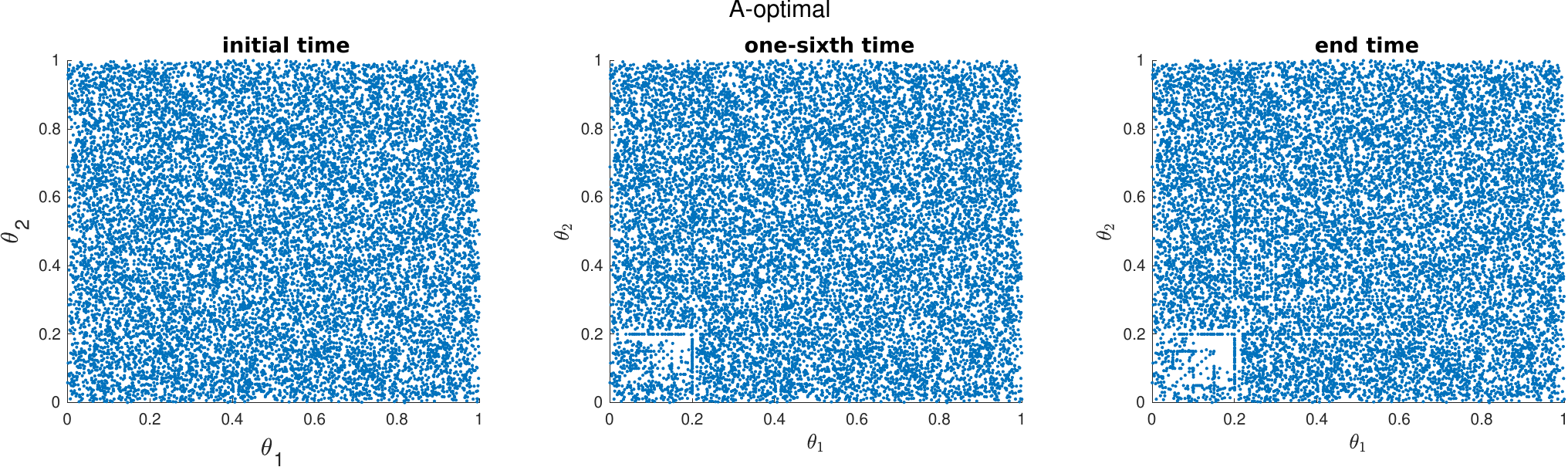}
\caption{Evolution of particle gradient flow Algorithm~\ref{alg: particle flow} under A-optimal criterion.}
\label{fig: A evolution new}
\end{figure}

Finally, we discuss the performance of Algorithm~\ref{alg: particle flow} on D-optimal design. In the gradient magnitude Figure~\ref{fig: D gradient new}, the active particle regions are two diagonal blocks which are separated at around $(0.25, 0.25)$, coinciding with the medium bump. Intense velocity in the diagonal is due to the strong signal received when source-detection pair coincides.
Three snapshots along the evolution are presented in Figure~\ref{fig: D evolution new}. Note that the samples concentrate in the two blocks in which $\theta_1$ and $\theta_2$ take on complimentary locations. 

 We conclude with a discussion of physical interpretations. Alongside Figures~\ref{fig: D gradient new} and \ref{fig: D evolution new}, slow particle regions (the two side blocks) correspond to the desired sensor locations, while active nodes in the complimentary diagonal blocks tend to leave their active regions. This indicates that OED favors stability over volatility: large gradients do not necessarily mark optimal sensor spots, whereas slow-moving region captures particles and stabilizes the design. For practical data acquisition, placing the source and detector on opposite sides of boundary [0,1] can improve measurement quality and reduce uncertainty.
\begin{figure}[!htb]
\centering
\includegraphics[scale=0.55]{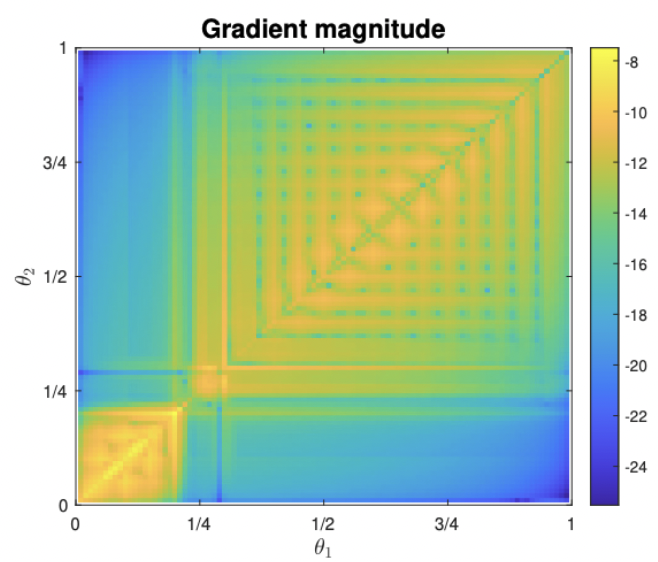}
\caption{Gradient magnitude under the D-optimal objective. }
\label{fig: D gradient new}
\end{figure}

\begin{figure}[!htb]
\centering
\includegraphics[scale=0.42]{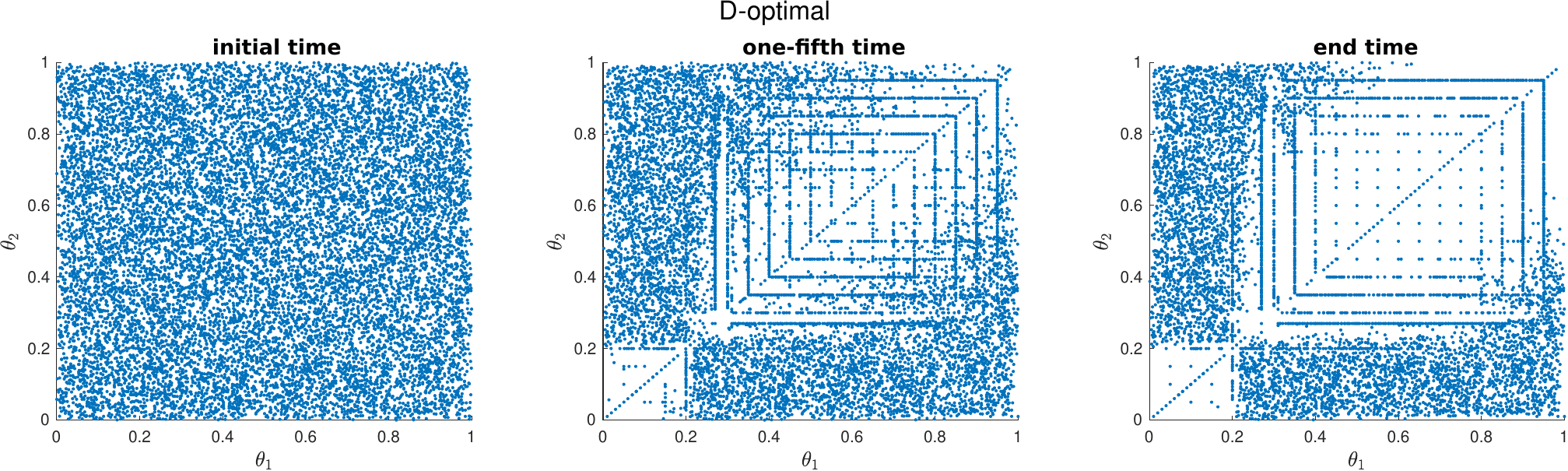}
\caption{Evolution of particle gradient flow Algorithm~\ref{alg: particle flow} under D-optimal criterion.}
\label{fig: D evolution new}
\end{figure}

Following up on the intriguing D-optimal design layout in Figure \ref{fig: D evolution new}, where particles concentrate in two local blocks, we aim to further incorporate a sparsity-promoting feature into the standard design criterion. 
Here, we use the terms ``sparsity" and ``concentration" interchangeably to describe the same effect: the clustering of particles in specific regions. The sparsity feature would enhance the practicality of our design algorithm and guide the data allocation procedure. 

Specifically, we modify the D-optimal criterion \eqref{eqn: cont D-opt} as
\begin{equation}
\label{eqn: sparse D-opt}
\max_{\rho \in \text{Pr}_2(\Omega)} F^D[\rho] +  \,\mathcal{R}[\rho],
\end{equation}
where the concentration reward is given by 
\begin{equation}
\label{eqn: kernel}
\mathcal{R}[\rho] = \int_{\Omega}\int_{\Omega} \text{exp}\left(-\frac{\| \theta - \theta'\|_2^2}{0.01} \right) \rho(\theta) \, \rho(\theta')\, \dd \theta \dd \theta'.
\end{equation}
Intuitively, this kernel reward encourages distributions in which particles nearby stick together, since contributions to the objective are the largest when particles are near one another. Thus, it serves as a sparsity-inducing regularizer.

In Figure \ref{fig: D evolution new sparse}, we implement Algorithm~\ref{alg: particle flow} to maximize the new design criterion with sparsity control \eqref{eqn: sparse D-opt}. The design measure $\rho$ starts from the two dense local blocks observed in the end pattern of Figure~\ref{fig: D evolution new}. We compare the resulting OEDs \textbf{with} and \textbf{without sparsity control}. It is evident that the design with the sparsity control \eqref{eqn: kernel} forms even denser concentrations, as the kernel in \eqref{eqn: kernel} creates a self-attracting field, whereas the design without this encouragement remains the original regions. Importantly, both design procedures (A) and (B) of Figure \ref{fig: D evolution new sparse} achieve the same D-optimal objective value, as shown in the right panel, indicating that sparsity regularization can operate independently of the pursuit of D-optimality.

\begin{figure}[!htb]
\centering
\subfloat[{\textbf{With sparsity control}}]{\includegraphics[scale=0.42]{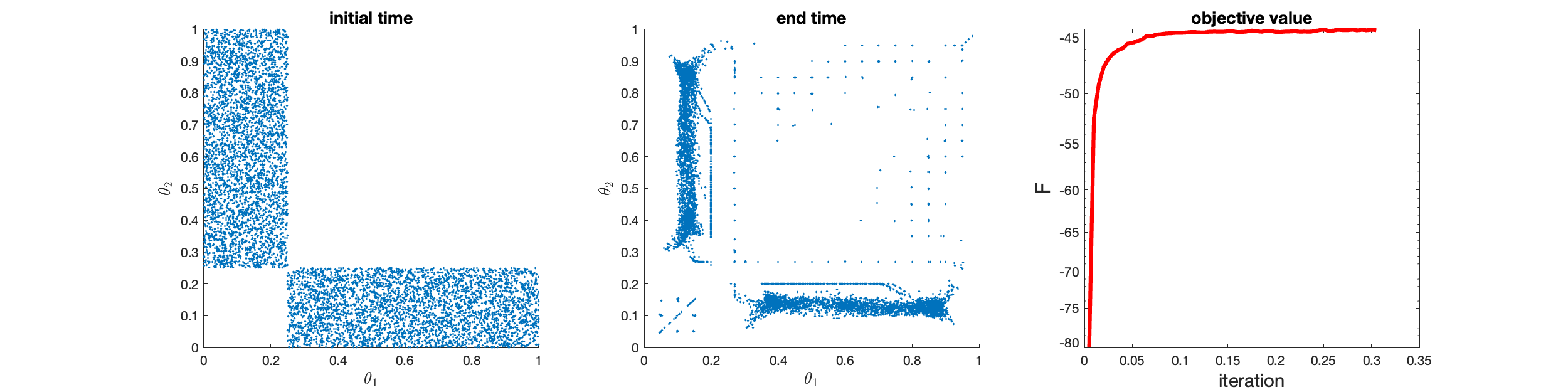}}\\
\subfloat[{\textbf{Without sparsity control}}]{\includegraphics[scale=0.42]{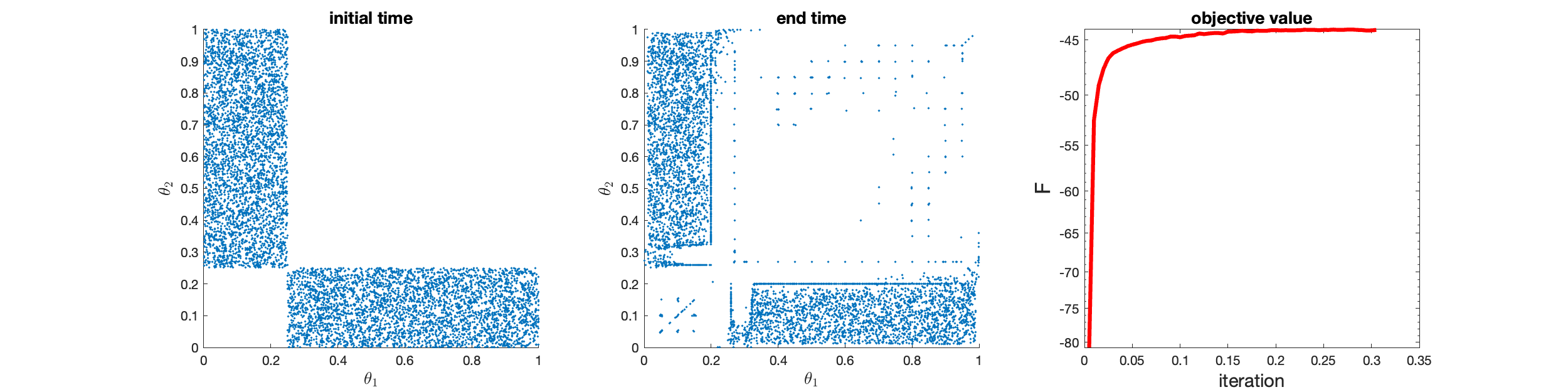}}
\caption{Evolution of particle gradient flow Algorithm~\ref{alg: particle flow} under D-optimal criterion with local initialization.}
\label{fig: D evolution new sparse}
\end{figure}

%% file: sections/7-conclusions.tex
\section{Conclusions}

As computational techniques involving optimal transport and Wasserstein gradient flow become more mature, they offer the opportunity to deal with infinite-dimensional probability measure space, enabling a new and wider range of  applications. 
The optimal experimental design (OED) problem in continuous design space is one such example, offering an important generalization over more traditional discrete experimental design.

The move from finite-dimensional Euclidean space to the infinite-dimensional probability manifold results in a more challenging optimization problem. 
We use newly available Wasserstein gradient flow techniques to recast the continuous OED problem. 
In particular, the gradient flow on measure space is mapped to gradient descent on a discrete set of particles representing the distribution in the Euclidean space. 
Algorithm~\ref{alg: particle flow} can be applied to solve the continuous OED. 
Moreover, we have provided the first criticality condition and basic convexity analysis under the A- and D-optimal design criteria. As a proof of concept, we assessed the algorithm's performance on the EIT problem, observing convergence of Algorithm~\ref{alg: particle flow} to distributions that reveal interesting design knowledge on specific EIT media examples. 

The present work opens the door to many additional questions, including the following.
\begin{enumerate}
\item Tensor structure in particle gradient flow Algorithm~\ref{alg: particle flow}. If the design space $\Omega$ is high dimensional, it might be possible to decompose the particle update in Algorithm~\ref{alg: particle flow} and follow a multi-modal scheme. How to take advantage of such tensor structure to improve efficiency of the optimization process is one interesting direction to pursue. 
\item Sensitivity to noise. It would be interesting to study sensitivity of the continuous OED optimizer  to the noise encoded in the objective function and data.
\item Explicit error bound in the simulation. The results available currently do not constitute a rigorous numerical analysis, though we provided a roadmap in Section~\ref{subsec: simulation error}. It would be interesting to fill in the missing technical arguments.  

\item Nonlinear optimal design.  
This paper addresses the optimal design problem only for the linear inverse problem. 
Most physical models do not have a linear relationship between the inferred quantity and data observations, so study of  OED in the convoluted nonlinear setting is a useful question for future study. 
\item Incorporating sparsity in the continuous design framework.
For data allocation, OED patterns that are more concentrated or sparse across most area can provide useful guidance for practitioners.
\end{enumerate}

\section*{Acknowledgements}
R.J., M.G. and Q.L. acknowledge the support from NSF-DMS-2308440. All authors of this paper are further supported by NSF-DMS-2023239. M.G. is supported by NSF-DMS-2012292.